\documentclass[12pt]{amsart}
\usepackage{amsmath}
\usepackage{amssymb}
\newtheorem{theorem}{Theorem}[section]
\newtheorem{lemma}[theorem]{Lemma}
\newtheorem{corollary}[theorem]{Corollary}
\newtheorem{proposition}[theorem]{Proposition}
\newtheorem{remark}[theorem]{Remark}
\newtheorem{definition}[theorem]{Definition}
\newtheorem{example}[theorem]{Example}

\numberwithin{equation}{section}

\def\A{\mathcal A}  
\def\B{\mathcal B}
\def\C{\mathcal C}
\def\D{\mathcal D}
\def\K{\mathcal K}

\def\F{\mathcal F}
\def\Q{\mathcal Q}
\def\P{\mathcal P}
\def\X{\mathcal X}

\def\Y{\mathcal Y}

\def\id{1}
\def\H{\mathcal H}  

\def\ZZ{\mathbb Z} 
\def\CC{\mathbb C} 
\def\RR{\mathbb R} 

\def\essmod{\mathrm{essmod}}
\def\unmod{\mathrm{unmod}}
\def\mod{\mathrm{mod}}
\def\ptp{\mathbin{\widehat{\otimes}}}
\def\iso{\cong}

\newcommand{\Tor}{\rm Tor} 
 
\newcommand{\ext}{\operatorname{Ext}}
\newcommand{\Xer}{\mathop{\rm Ker}~ } 
\newcommand{\im}{\mathop{\rm Im}}
\newcommand{\from}{\leftarrow}

\begin{document}
\title[The higher-dimensional amenability of tensor products of algebras]{The higher-dimensional amenability of tensor products of Banach algebras}

\author{Zinaida A. Lykova}

\address{Zinaida A. Lykova, School of Mathematics and Statistics,  
Newcastle University, 
Newcastle upon Tyne, NE\textup{1} \textup{7}RU, UK}

\email{Z.A.Lykova@newcastle.ac.uk} 

\date{April 29, 2009}

\begin{abstract} We investigate the higher-dimensional amenability of
tensor products $\A \ptp \B$ of Banach algebras $\A$ and $\B$.
We prove that the weak bidimension $db_w$ of the tensor product $\A \ptp \B$ of Banach algebras $\A$ and $\B$ with bounded approximate identities satisfies 
\[
db_w \A \ptp \B = db_w \A + db_w \B.
\]
We show that it cannot be extended to arbitrary Banach algebras. For example, for  a biflat
 Banach algebra $\A$ which has a left or right, but not two-sided,  bounded approximate identity, we have   
$db_w \A \ptp \A \le 1$ and $db_w \A + db_w \A =2.$
We  describe explicitly the continuous Hochschild cohomology $\H^n(\A \ptp \B, (X \ptp Y)^*)$ and the cyclic  cohomology $\H\C^n(\A \ptp \B)$ of certain tensor products $\A \ptp \B$ of Banach algebras $\A$ and $\B$  with bounded approximate identities; here $(X \ptp Y)^*$ is the dual bimodule of the tensor product of essential Banach bimodules
$X$  and $Y$ over  $\A$ and $\B$ respectively.

\noindent 2000 {\it Mathematics Subject Classification:} 
Primary 55U25, 19D55, 43A20.
\end{abstract}


\keywords{Hochschild cohomology, cyclic cohomology, $C^*$-algebra, semigroup algebra.}

\maketitle

\section{INTRODUCTION}

Hochschild  cohomology groups of Banach and $C^*$ algebras play an important role in $K$-theory \cite{Cu} and in noncommutative geometry \cite{Co2}. However, it is very difficult to describe explicitly non-trivial higher-dimensional cohomology for Banach algebras. In this paper we consider Hochschild  cohomology groups  of tensor products $\A \ptp \B$ of Banach algebras $\A$ and $\B$  with bounded approximate identities.

A Banach algebra $\A$ such that $ {\mathcal H}^1(\A,X^*) =\{ 0 \}$
for all Banach $\A$-bimodules $X$ is called {\it amenable} \cite{BEJ1}.
B. E. Johnson proved in \cite[Proposition 5.4]{BEJ1} that $\A \ptp \B$
is amenable if the Banach algebras $\A$ and $\B$ are amenable. We determine relations between the higher-dimensional amenability of Banach algebras $\A$ and $\B$,  and the  higher-dimensional amenability of their tensor product algebra $\A \ptp \B$.
Virtual diagonals and higher-dimensional amenability of Banach algebras  were investigated in a paper of E.G. Effros and  A. Kishimoto \cite{EfKi87} for unital algebras  and in papers of 
A.L.T. Paterson and R.R. Smith \cite{Pat2,PaSm} in the non-unital case.
Recall that, for any $n \ge 1$, a Banach algebra $\A$ is called $n$-{\it amenable} if  the continuous Hochschild cohomology $\H^n(\A, X^*) =\{0 \}$ for every Banach $\A$-bimodule $X$. The {\it  weak bidimension} of a Banach algebra $\A$ is
\[
db_w \A = \inf~ \{n: {\mathcal H}^{n+1} (\A, X^*) = \{ 0 \} \; {\rm for \; all \; Banach} \; \A\text{-bimodule} \; X \}
\]
(see \cite{Sel95}).  It is clear that  a Banach algebra $\A$ is $n$-amenable if and only if $db_w \A = n-1$, and that $\A$ is amenable if and only if $db_w \A = 0$.

In 1996  Yu. Selivanov announced in Remark 4 \cite{Sel96} without proof that the weak bidimension $db_w$ of the tensor product $\A \ptp \B$ of Banach algebras $\A$ and $\B$ with bounded approximate identities satisfies 
\[
db_w \A \ptp \B = db_w \A + db_w \B.
\]
In 2002 he gave in \cite[Theorem 4.6.8]{Sel02} a proof of the formula  in the particular case of algebras with identities, and his proof depends heavily on the existence of identities. In Theorem \ref{dbwAptpB} of this paper we prove that the formula is correct for  algebras with bounded approximate identities (b.a.i.). We show further that the formula does {\em not} hold for algebras with only 1-sided b.a.i., nor for algebras no b.a.i.
To this end we need to develop homological tools for algebras with bounded approximate identities and flat essential modules. The well-known trick  adjoining  of an identity to the algebra does not work for  the tensor product of algebras. 
The homological properties of the tensor product algebras $\A {\mathbin{\widehat{\otimes}}} \B$ and $\A_+ {\mathbin{\widehat{\otimes}}} \B_+$ are different; here
$\A_+$ is the Banach algebra obtained by adjoining an identity to $\A$. In Theorem \ref{AptpAone-side-bai} we show that, for a biflat
 Banach algebra $\A$ which has a left or right, but not two-sided,  bounded approximate identity, we have  
$db_w \A \ptp \A \le 1$ and $db_w \A_+ \ptp \A_+ = 2 db_w \A = 2$. For example, for the algebra $\K(\ell_2 \ptp \ell_2)$ of compact operators on $ \ell_2 \ptp \ell_2$ and any integer $n \ge 1$,  
\[
db_w [\K(\ell_2 \ptp \ell_2)]^{\ptp n} \le 1 \; \;{\rm and} \;\;
db_w [\K(\ell_2 \ptp \ell_2)_+]^{\ptp n}=n.
\]

We prove that, for  the tensor product $\A \ptp \B$ of Banach algebras $\A$ and $\B$ with bounded approximate identities, for the tensor product $X \ptp Y$ of essential Banach bimodules $X$ and $Y$ over $\A$ and $\B$ respectively and for $n \ge 0$, up to topological isomorphism, 
\[ \H^n(\A \ptp \B, (X \ptp Y)^*) = 
H^n(({\mathcal C}_{\sim}(\A, X) \ptp {\mathcal C}_{\sim}(\B, Y))^*), \]
where ${\mathcal C}_{\sim}(\A, X)$ and ${\mathcal C}_{\sim}(\B, Y)$ are the standard homological chain complexes.

We  describe explicitly the continuous Hochschild cohomology 
\newline $\H^n(\A \ptp \C, (X \ptp Y)^*)$
and the cyclic  cohomology $\H\C^n(\A \ptp \C)$ of certain tensor products of Banach algebras. 
For example, in Corollary \ref{ho-coho-L1-B}  we prove that, for
an amenable Banach algebra $\C$,
$$\H^n(L^1(\RR_+^k)\ptp \C, (L^1(\RR_+^k)\ptp \C)^*) \iso
 {\bigoplus\nolimits^{k \choose n}} \left[L^1(\RR_+^k) \ptp \left(\C/[\C,\C] \right) \right]^* $$
if $n\leq k.$ In Corollary \ref{ho-coho-A-B(H)} we show that all continuous cyclic type cohomology groups are trivial for
a Banach algebra ${\mathcal D}$ belonging to
one of the following classes:
\newline {\rm (i)} $\D=\ell^1(\ZZ_+^k) \ptp \C$, where 
 $\C$ is a $C^*$-algebra without non-zero bounded traces;
or
\newline {\rm (ii)} $\D=L^1(\RR_+^k) \ptp \C$, where 
 $\C$ is a $C^*$-algebra without non-zero bounded traces.

\section{DEFINITIONS AND NOTATION}

 We recall some notation and terminology used in  homological theory. 
These can be found in any textbook on homological algebra, for instance, 
MacLane \cite{Mac} for the pure algebraic case 
and  Helemskii \cite{He0} for the continuous case.

For a Banach algebra $\A$, let $X$ be a Banach $\A$-bimodule and let $S$ be a subset of $\A$.  Then 
$S^2 $ is defined to be the linear span of the set 
$\{ a_1 \cdot a_2 : \; a_1, a_2 \in S \} $, $SX$ is the linear span of the
set $\{ a \cdot x : \; a \in S,\; x \in X \} $ and
$\overline{SX} $ is the closure of  $SX$ in $X$. Expressions like  
$\overline{XS}$ and $\overline{SXS}$ have similar meanings. 
A Banach $\A$-bimodule $X$ is called {\it essential} if $X =
\overline{\A X \A}$.
Let  $I$ be a closed two-sided ideal with a bounded approximate identity;
then, by an extension of the Cohen factorization theorem,
 $\overline{IX}= \{ b \cdot x: b \in I,$ \hbox{$ x \in X \}$.}
These facts may be found in
 \cite{Hew, CF-T}; see also 
\cite[Proposition 1.8]{BEJ1} and \cite[Theorem 5.2.2]{Pa}.

The category  of  Banach spaces and 
continuous linear operators is denoted by 
${\mathcal Ban}$. For a Banach algebra $\A$,
the category of  left [essential] \{unital\} Banach    $\A$-modules is denoted by 
$\A$-$\mod$ [$\A$-$\essmod$] \{$\A$-$\unmod$\},
the category of  right [essential] \{unital\}  Banach    $\A$-modules is denoted by 
$\mod$-$\A$ [$\essmod$-$\A$] \{$\unmod$-$\A$\}
 and the category of  [essential] \{unital\}  Banach  $\A$-bimodules is denoted by $\A$-$\mod$-$\A$ [$\A$-$\essmod$-$\A$] \{$\A$-$\unmod$-$\A$\}.

For a Banach space $E$, we will denote by  $E^*$ 
 the dual space of $E$. In the case of Banach algebras $\A$,
for a Banach $\A$-bimodule $X$, $X^*$ is the Banach $\A$-bimodule dual to $X$ with the module multiplications given by
\[ (a \cdot f)(x) = f(x \cdot a), \; (f \cdot a)(x) = f(a \cdot x)\;\;(a \in \A, \; f \in X^*, \; x \in X ). 
\]
 A {\it chain complex} $\X$ in ${\mathcal Ban}$
is a family of Banach spaces $X_n$ and continuous linear maps
$d_n$ (called {\em boundary maps}) 
$$ \dots \stackrel {d_{n-2}} {
\longleftarrow} X_{n-1} \stackrel {d_{n-1}} { \longleftarrow}  X_{n} 
 \stackrel {d_{n}} { \longleftarrow} X_{n+1} \stackrel {d_{n+1}} 
{ \longleftarrow} X_{n+2}
\longleftarrow  \cdots   $$
such that $\im d_n \subset {\rm Ker}~   d_{n-1}$. 
The subspace  $\im d_{n}$ of $X_{n}$  is denoted by $B_n(\X)$ and 
its elements are called {\it boundaries}. The Banach
subspace ${\rm Ker}~   d_{n-1}$ of $X_{n}$ is denoted by $Z_n(\X)$ and 
its elements are {\it  cycles}. The {\it  homology groups of $\X$} 
are defined by
 $H_n(\X) = Z_n(\X)/B_n(\X)$.
As usual, we will often drop the subscript $n$ of $d_n$. If there 
is a need to distinguish between various boundary maps on various 
chain complexes, we will use subscripts, that is, we will denote 
the boundary maps on $\X$ by $d_{\X}$.
 A chain complex $\X$ is called {\it bounded} if $X_{n} = \{ 0 \}$ 
whenever $n$ is less than a certain fixed integer $N \in \ZZ$.

 Let ${\K}$ be one of the above categories of Banach 
$\A$-modules and morphisms. For $X, Y \in {\K}$, 
the Banach space of morphisms from $X$ to $Y$ is denoted by 
$h_{\K} (X, Y)$. We shall abbreviate 
$ h_{\A\hbox{-} \mod}$ to $ h_{\A}$ and 
 $ h_{\A\hbox{-} \mod \hbox{-}\A}$ to $_{\A}h_{\A}$.

 A complex of   Banach  $\A$-modules and morphisms
is called {\it  admissible} if it 
splits as a complex of Banach  spaces \cite[III.3.11]{He0}.
 A complex of   Banach  $\A$-modules and morphisms
is called {\it weakly admissible} if its dual complex  
splits as a complex of Banach  spaces \cite[III.3.11]{He0}.

A module $P \in \K$ is called {\it projective in} $\K$ if, for each module $Y  \in \K$ and each epimorphism of modules $\varphi \in h_{\K}(Y,P)$ such that $\varphi$ has a right inverse as a morphism of Banach spaces, there exists a  right inverse morphism of Banach  modules
from $\K$.

 For $Y \in \K$, a complex 

\vspace*{0.2cm}
\hspace{1.5cm}
$0  \longleftarrow Y
\stackrel {\varepsilon} { \longleftarrow}  P_0 
 \stackrel {\phi_0} { \longleftarrow} P_1 \stackrel {\phi_1} 
{ \longleftarrow} P_2 \longleftarrow  \cdots  
\hfill {(0 \leftarrow Y \stackrel {\varepsilon} { \longleftarrow}  (\P,\phi))} $ 
\vspace*{0.2cm}
\newline is called a {\it   projective resolution  of  $Y$ in  $\K$ } if it is admissible and all the modules in ${\mathcal P}$ are  
projective in $\K$ \cite[Definition III.2.1]{He0}. 

  We shall denote the $n$th cohomology 
of the complex 
 $ h_{\K}(\P, X)$ where $X \in {\K}$ by 
 $ {\rm \ext}^n_{\K}~(Y, X)$. We shall abbreviate 
$ {\rm \ext}^n_{\A \hbox{-} \mod}$ to $ {\rm \ext}^n_{\A}$ and 
 $ {\rm \ext}^n_{\A \hbox{-} \mod \hbox{-} \A}$ to $ {\rm \ext}^n_{\A^e}$.

Further ${\mathbin{\widehat{\otimes}}}$ is the projective tensor product of Banach spaces \cite{CLM}, \cite[II.4.1]{He0},  $
{\mathbin{\widehat{\otimes}}}_\A$ is the projective tensor product
of   left and right Banach $\A$-modules,
${\mathbin{\widehat{\otimes}}}_{\A-\A}$ is the projective tensor product
of  Banach $\A$-bimodules (see \cite{Rie}).
Note that by $ X^{\ptp 0} \ptp Y $ we mean $Y$, by $ X^{\ptp 1}$ we mean $X$
and by $ X^{\ptp n}$ we mean the $n$-fold projective tensor power $ X\ptp \dots \ptp X $ of $X$.

 For any Banach algebra $\A$, not necessarily unital, $\A_+$ is the Banach algebra obtained by adjoining an identity to $\A$.
 For a Banach  algebra $\A$, the algebra 
$\A^e = \A_+ {\mathbin{\widehat{\otimes}}} \A_+^{op}$
is called the {\it 
 enveloping algebra of $\A$}, where $\A_+^{op}$ is the {\it  opposite algebra}
of $\A_+$ with  multiplication $a \cdot b = ba$.

 A module $Y \in \A$-$\mod$ is called 
{\it flat } if for any admissible complex ${\X}$ of right 
Banach  $\A$-modules the complex ${\X} {\mathbin{\widehat{\otimes}}}_{\A} Y$
is exact. A module $Y \in  \A$-$\mod$-$\A$ is called 
{\it flat } if for any admissible complex ${\X}$ of  
Banach  $\A$-bimodules the complex ${\X} {\mathbin{\widehat{\otimes}}}_{\A^e} Y$
is exact.

   For $X \in {\K}$, a complex 

\vspace*{0.2cm}
\hspace{1.5cm}
$0  \longleftarrow X
\stackrel {\varepsilon} { \longleftarrow}  Q_0 
 \stackrel {\phi_0} { \longleftarrow} Q_1 \stackrel {\phi_1} 
{ \longleftarrow} Q_2
\longleftarrow  \dots  
\hfill {(0  \leftarrow X \stackrel {\varepsilon} { \longleftarrow}{(\Q, \phi)})} $
\vspace*{0.2cm}
\newline is called a {\it  pseudo-resolution of $X$ in  ${\K}$ } 
if it is weakly admissible,
 and a {\it flat pseudo-resolution of $X$ in  ${\K}$ } if, in addition, all the modules in ${\Q}$ are flat in ${\K}$.

 For $X  \in $ $\mod$-$\A$  and $Y \in \A$-$\mod$, we shall denote by  $ \Tor^{\A}_{\it n}(X, Y)$ the $n$th homology 
of the complex  $ X \ptp_{\A}{\mathcal P}$, where 
$(0 \leftarrow Y \leftarrow {\mathcal P})$ is a  projective resolution in
$\A$-$\mod$,  \cite[Definition III.4.23]{He0}.
For $X, Y  \in \A$-$\mod$-$\A$, the Tor-spaces are defined by using the standard identifications $\A$-$\mod$-$\A \cong \A^e$-$\unmod
\cong \unmod$-$\A^e$.

Throughout the paper  $\id_X: X \to X$ denotes the identity
operator and $\iso$ denotes an isomorphism of Banach spaces. 
Given a Banach space $E$ and a chain complex $(\X, d)$ in 
 ${\mathcal Ban}$, we can form the chain complex 
$E \ptp \X$ of the Banach spaces $E \ptp X_n$ and 
boundary maps $\id_E \otimes d$.

We recall the definition of  the tensor product $\X \ptp \Y$ of bounded complexes $\X$ and $\Y$ in ${\mathcal Ban}$ which can be found
  in \cite[Definitions II.5.25]{He0}.

\begin{definition}\label{XptpY}
 Let $\X$, $\Y$ 
be chain complexes in ${\mathcal Ban}$:
$$ 0 \stackrel {\phi_{-1}} {\longleftarrow} X_{0} \stackrel {\phi_{0}} { \longleftarrow}  X_{1} 
 \stackrel {\phi_{1}} { \longleftarrow} X_{2} \stackrel {\phi_{2}} 
{ \longleftarrow} X_{3} \longleftarrow  \cdots   $$
and 
$$ 0 \stackrel {\psi_{-1}} {\longleftarrow} Y_{0} \stackrel {\psi_{0}} { \longleftarrow}  Y_{1} 
 \stackrel {\psi_{1}} { \longleftarrow} Y_{2} \stackrel {\psi_{2}} 
{ \longleftarrow} Y_{3} \longleftarrow  \cdots .  $$
The tensor product $\X \ptp \Y$ of bounded complexes $\X$ and $\Y$ in ${\mathcal Ban}$  is the chain complex 
\begin{equation}
\label{complexXptpY}
 0 \stackrel{\delta_{-1}}{\longleftarrow} (\X \ptp \Y)_{0} \stackrel {\delta_{0}} {\longleftarrow}  (\X \ptp \Y)_{1} 
 \stackrel{\delta_{1}}{\longleftarrow} (\X \ptp \Y)_{2} \stackrel {\delta_{2}}{\longleftarrow} (\X \ptp \Y)_{3} \longleftarrow  \cdots,   
\end{equation}
where
$$(\mathcal X {\mathbin{\widehat{\otimes}}}{\mathcal Y})_n =  
\displaystyle{\bigoplus_{m+q=n}} X_m \ptp Y_q $$
and 
$$ \delta_{n-1} (x \otimes y) = \phi_{m-1}(x) \otimes y + (-1)^m x \otimes   \psi_{q-1}(y),$$
$x \in X_m, y \in Y_q$ and $m+q=n$.
\end{definition}

For Banach spaces $E$ and $F$,  let ${\mathfrak B}(E,F)$ be the Banach space of all continuous linear operators from $E$ to $F$.


\section{PSEUDO-RESOLUTIONS IN CATEGORIES OF BANACH MODULES}

\begin{lemma}\label{WeakAdmis} Let
\begin{equation}\label{weakly-admis-long}
\dots \stackrel {d_{n-2}} {
\longleftarrow} X_{n-1} \stackrel {d_{n-1}} { \longleftarrow}  X_{n} 
 \stackrel {d_{n}} { \longleftarrow} X_{n+1} \stackrel {d_{n+1}} 
{ \longleftarrow} X_{n+2}
\longleftarrow  \cdots  
\end{equation}
be a weakly admissible complex of Banach spaces and continuous linear operators. 
 Then,  for every Banach space $Y$,  the sequence
\begin{equation}\label{weakly-admis-Y-otimes-long}
\dots 
\longleftarrow Y {\mathbin{\widehat{\otimes}}} X_{n-1} \stackrel {\id_Y \otimes d_{n-1}}
 { \longleftarrow} Y {\mathbin{\widehat{\otimes}}}  X_{n} 
 \stackrel {\id_Y \otimes d_{n}} { \longleftarrow} 
Y {\mathbin{\widehat{\otimes}}} X_{n+1}\stackrel {\id_Y \otimes d_{n+1} } 
{ \longleftarrow} Y {\mathbin{\widehat{\otimes}}} X_{n+2}
\longleftarrow  \cdots   
\end{equation}
is  weakly admissible.
\end{lemma}

\begin{proof}  By assumption the complex (\ref{weakly-admis-long})
is weakly admissible. Therefore, for every $n$,  there is a bounded linear operator 
\[
s_n :  X_{n+1}^* \to  X_{n}^*, 
\]
such that 
$d_{n-1}^* \circ s_{n-1} + s_{n} \circ d_{n}^* =
\id _{X_{n}^*}$, where
$$ \dots \stackrel {d_{n-2}^*} {\longrightarrow} X_{n-1}^* \stackrel {d_{n-1}^*} { \longrightarrow}  X_{n}^* 
 \stackrel {d_{n}^*} { \longrightarrow} X_{n+1}^* \stackrel {d_{n+1}^*} { \longrightarrow} X_{n+2}^* \longrightarrow  \cdots   $$ 
Further we will use the well-known isomorphism 
\cite[Theorem 2.2.17]{He0}
\[
{\mathfrak B}(Y, X_n^*)  \rightarrow (Y {\mathbin{\widehat{\otimes}}} X_n)^*: \phi \mapsto \Phi_\phi
\]
where $\Phi_\phi(y \otimes x) = [\phi(y)](x); y \in Y, x \in X_n$ and
\[
 (Y {\mathbin{\widehat{\otimes}}} X_n)^* \rightarrow {\mathfrak B}(Y, X_n^*): f \mapsto \phi_f
\]
where $[\phi_ f(y)](x) = f(y \otimes x); \; y \in Y, x \in X_n.$
One can see that, for  $ f \in (Y {\mathbin{\widehat{\otimes}}} X_n)^*  $, we have the following 
$\phi_ f \in {\mathfrak B}(Y, X_n^*)  $,~~
$s_{n-1} \circ \phi_ f \in {\mathfrak B}(Y, X_{n-1}^*) $ and $\Phi_{s_{n-1} \circ \phi_ f } \in 
(Y {\mathbin{\widehat{\otimes}}} X_{n-1})^*  $. We define a map  
\[ 
\gamma_{n-1}:
(Y {\mathbin{\widehat{\otimes}}} X_n)^* \rightarrow (Y {\mathbin{\widehat{\otimes}}} X_{n-1})^*
\] 
by $ \gamma_{n-1}(f)
 = \Phi_{ s_{n-1} \circ \phi_ f}$ for  $ f \in (Y {\mathbin{\widehat{\otimes}}} X_n)^* $.
It is easy to see that  $ \gamma_{n-1}$  is a continuous linear
operator.
One can check that
\[
 \gamma_{n} \circ (\id_Y \otimes d_{n})^*  +  (\id_Y \otimes d_{n-1})^* \circ \gamma_{n-1}  = \id_{(Y {\mathbin{\widehat{\otimes}}} X_{n})^*},
\] 
where
\begin{equation}\label{dual-Y-otimes-long}
\dots \stackrel {(\id_Y \otimes d_{n-2})^*} {
\longrightarrow} (Y {\mathbin{\widehat{\otimes}}} X_{n-1})^* \stackrel {(\id_Y \otimes d_{n-1})^*}
 { \longrightarrow} (Y {\mathbin{\widehat{\otimes}}}  X_{n})^* 
 \stackrel {(\id_Y \otimes d_{n})^*} { \longrightarrow} 
(Y {\mathbin{\widehat{\otimes}}} X_{n+1})^* \longrightarrow \cdots   
\end{equation}

The result follows.
\end{proof}

We shall denote $\gamma_{n-1}$ from the previous  lemma by
$\gamma_{(s,X_n,Y)}$.

\begin{lemma}\label{CommutativeWeakAdmis} Let
\begin{equation}\label{commutative-weakly-admis-long}
\dots \stackrel {d_{n-2}} {
\longleftarrow} X_{n-1} \stackrel {d_{n-1}} { \longleftarrow}  X_{n} 
 \stackrel {d_{n}} { \longleftarrow} X_{n+1} \stackrel {d_{n+1}} 
{ \longleftarrow} X_{n+2}
\longleftarrow  \cdots  
\end{equation}
be a weakly admissible complex of Banach spaces and continuous operators. 
 Then,  for Banach spaces $Y$ and $Z$ and a bounded linear operator $\delta: Y \to Z$,  the following diagram
\begin{equation}\label{commutative-weakly-admis-Y-otimes-long}
\begin{array}{ccccccccc}
\dots & 
\stackrel{\gamma_{(s,X_{n-1},Z)}}{\longleftarrow} 
& (Z {\mathbin{\widehat{\otimes}}} X_{n-1})^* & 
\stackrel{\gamma_{(s,X_{n},Z)}}{\longleftarrow} &
(Z {\mathbin{\widehat{\otimes}}}  X_{n})^* &
\stackrel {\gamma_{(s,X_{n+1},Z)}} {\longleftarrow} &
(Z {\mathbin{\widehat{\otimes}}} X_{n+1})^* &\cdots   \\

\dots & ~ &
\vcenter{\llap{$\scriptstyle{(\delta \otimes \id_{X_{n-1}})^*}~~$}}{\downarrow}  & ~ &
\vcenter{\llap{$\scriptstyle{(\delta \otimes \id_{X_{n}})^*}~~$}}{\downarrow}  & ~ &
\vcenter{\llap{$\scriptstyle{(\delta \otimes \id_{X_{n+1}})^*}~~$}}{\downarrow} \\

\dots &
\stackrel {\gamma_{(s,X_{n-1},Y)}} { \longleftarrow} &
(Y {\mathbin{\widehat{\otimes}}} X_{n-1})^* &
\stackrel {\gamma_{(s,X_{n},Y)}} { \longleftarrow}  &(Y {\mathbin{\widehat{\otimes}}}  X_{n})^* &
\stackrel {\gamma_{(s,X_{n+1},Y)}} { \longleftarrow} &
(Y {\mathbin{\widehat{\otimes}}} X_{n+1})^* &\cdots   \\
\end{array}
\end{equation}
is commutative.
\end{lemma}

\begin{proof} We shall show that for every $ f \in (Z {\mathbin{\widehat{\otimes}}} X_n)^* $,
\[ (\delta \otimes \id_{X_{n-1}})^* \circ \gamma_{(s,X_{n},Z)}=
\gamma_{(s,X_{n},Y)} \circ (\delta \otimes \id_{X_{n}})^*.
\]
We shall use notations from the previous lemma.
Note that for every $x \in X_{n-1}$, $ y \in Y$,
\[
[\phi_ f (\delta y)] (x)=
f (\delta y \otimes x)
\]
and
\[
[\phi_{f \circ (\delta \otimes \id_{X_{n-1}})}(y)](x)=
f \circ (\delta \otimes \id_{X_{n-1}}) (y \otimes x) = f (\delta y \otimes x).
\]

Thus, for all  $ y \in Y$,
\[ [\phi_ f (\delta y )] = [\phi_{f \circ (\delta \otimes \id_{X_{n-1}})}](y)
\]
and 
\[ [s_{n-1} \circ \phi_ f ](\delta y ) = [s_{n-1} \circ \phi_{f \circ (\delta \otimes \id_{X_{n-1}})}](y) = [s_{n-1} \circ \phi_{(\delta \otimes \id_{X_{n-1}})^*f}](y) .
\]
Therefore

\[ [(\delta \otimes \id_{X_{n-1}})^* \circ \gamma_{(s,X_{n},Z)} (f)](y \otimes x)= [\gamma_{(s,X_{n},Z)} (f)](\delta y \otimes x)
\]
\[ = \Phi_{s_{n-1} \circ \phi_ f}(\delta y \otimes x) = [s_{n-1} \circ \phi_ f](\delta y)(x)
\]
and
\[ [\gamma_{(s,X_{n},Y)} \circ (\delta \otimes \id_{X_{n}})^* (f)]( y \otimes x) = \Phi_{s_{n-1} \circ (\delta \otimes \id_{X_{n}})^* (f)}(y \otimes x)
\]
\[
= [s_{n-1} \circ \phi_{(\delta \otimes \id_{X_{n-1}})^* f}(y)](x)
= [s_{n-1} \circ \phi_ f ](\delta y )(x).
\]
\end{proof}

\begin{proposition}\label{tensor-product-w-a-resol} 
Suppose that  complexes of Banach spaces and continuous linear operators
\[
0  \longleftarrow X
\stackrel {\varepsilon_1} { \longleftarrow}  X_0 
 \stackrel {d_0} { \longleftarrow} X_1 \stackrel {d_1} 
{ \longleftarrow} X_2 \longleftarrow  \dots  \hspace*{1.5cm}
\hfill {(0 \leftarrow X 
\stackrel{\varepsilon_1}{\longleftarrow}{\mathcal X})} 
\]
and
\[
0  \longleftarrow Y
\stackrel {\varepsilon_2} { \longleftarrow}  Y_0 
 \stackrel {\tilde{d}_0} { \longleftarrow} Y_1 \stackrel {\tilde{d}_1} 
{ \longleftarrow} Y_2
\longleftarrow  \dots  \hspace*{1.5cm}
\hfill {(0 \leftarrow Y 
\stackrel{\varepsilon_2}{\longleftarrow}{\mathcal Y})} 
\]
are weakly admissible. Then
the complex ${0 \leftarrow X \hat{\otimes} Y \stackrel
{\varepsilon_1\otimes\varepsilon_2}{\longleftarrow} 
{\mathcal X} \hat{\otimes }{\mathcal Y}}$
is also weakly admissible.
\end{proposition}

\begin{proof}  By assumption the complexes
$ 0 \leftarrow X \stackrel{\varepsilon_1}{\longleftarrow}{\mathcal X}$ and $ 0 \leftarrow Y \stackrel{\varepsilon_2}{\longleftarrow} {\mathcal Y}$ are weakly admissible.
 Therefore, for every $n$,  there is a bounded linear operator 
\[
s_n :  X_{n+1}^* \to  X_{n}^*
\]
such that 
\[
d_{n-1}^* \circ s_{n-1} + s_{n} \circ d_{n}^* =
\id _{X_{n}^*},\;\; \varepsilon_1^* \circ s_{-1} + s_{0} \circ d_{0}^* =
\id _{X_{0}^*} \;\; \text{\rm and} \;\;  s_{-1} \circ \varepsilon_1^*  =
\id _{X^*}.
\]
Similarly,
 for every $n$,  there is a bounded linear operator 
\[
t_n :  Y_{n+1}^* \to  Y_{n}^*, 
\]
such that 
\[
\tilde{d}_{n-1}^* \circ t_{n-1} + t_{n} \circ \tilde{d}_{n}^* =
\id _{Y_{n}^*}, \;\; \varepsilon_2^* \circ t_{-1} + t_{0} \circ \tilde{d}_{0}^* =
\id _{Y_{0}^*} \;\; \text{\rm and} \;\;  t_{-1} \circ \varepsilon_2^*  =
\id _{Y^*}.
\]
Let us show that the complex 
\[
0  \longleftarrow X \hat{\otimes} Y
\stackrel {\varepsilon_1\otimes\varepsilon_2} { \longleftarrow}  
(\mathcal X \hat{\otimes}{\mathcal Y})_0 
 \stackrel {\delta_0} { \longleftarrow} (\mathcal X \hat{\otimes}{\mathcal Y})_1 \stackrel {\delta_1} 
{ \longleftarrow} (\mathcal X \hat{\otimes}{\mathcal Y})_2
\longleftarrow  \dots  
\]
is weakly admissible too. Recall that
\[(\mathcal X {\mathbin{\widehat{\otimes}}}{\mathcal Y})_n=  \displaystyle{\bigoplus_{m+q=n}} X_m \ptp Y_q
\]  
and  for each $x\otimes y \in X_m \ptp Y_q$, 
\[\; \delta_n (x \otimes y)= d_{m-1}(x) \otimes y + (-1)^m x \otimes \tilde{d}_{q-1}(y).\]

By virtue of Lemma \ref{WeakAdmis}, for $Y$ and $Y_p$, $p \ge 0$, the sequence
\[
0  \longrightarrow (X \hat{\otimes} Y_p)^*
\stackrel {(\varepsilon_1 \otimes \id_{Y_p})^*} { \longrightarrow}  
(X_0 \hat{\otimes} Y_p)^* 
 \stackrel {(d_{0} \otimes \id_{Y_p})^*} { \longrightarrow} 
(X_1 \hat{\otimes} Y_p)^* 
 \stackrel {(d_{1} \otimes \id_{Y_p})^*} { \longrightarrow} 
(X_2 \hat{\otimes} Y_p)^* 
 \cdots   \]
splits as a complex of Banach spaces, so that there exist bounded linear operators
\[
\gamma_{(s, X_n,Y_p)}:
(X_n \hat{\otimes} Y_p)^* \rightarrow (X_{n-1} \hat{\otimes} Y_p)^*
\]
such that
\begin{eqnarray}
\label{gamma-of-s}
\nonumber \; \; \gamma_{(s, X_{n+1},Y_p)} \circ (d_{n} \otimes \id_{Y_p})^*  +  ( d_{n-1}\otimes \id_{Y_p})^* \circ  \gamma_{(s, X_{n},Y_p)} = \id_{(X_{n} \hat{\otimes} Y_p)^*}\;\\
\;\; (\varepsilon_1 \otimes \id_{Y_p})^* \circ 
\gamma_{(s, X_{0},Y_p)} + \gamma_{(s, X_{1},Y_p)} \circ (d_{0} \otimes \id_{Y_p})^* = \id _{(X_{0} \hat{\otimes} Y_p)^*} \;\; \text{\rm and} \\
\nonumber \;\;  \gamma_{(s, X_{0},Y_p)} \circ (\varepsilon_1 \otimes \id_{Y_p})^*  = \id _{(X \hat{\otimes} Y_p)^*}.
\end{eqnarray}
Similarly, by virtue of Lemma \ref{WeakAdmis}, for $X$ and $X_p$, $p \ge 0$, the sequence
\[
0  \longrightarrow (X_p \hat{\otimes} Y)^*
\stackrel {( \id_{X_p} \otimes \varepsilon_2)^*} { \longrightarrow}  
(X_p \hat{\otimes} Y_0)^* 
 \stackrel {(\id_{X_p} \otimes  \tilde{d}_{0} )^*} { \longrightarrow} 
(X_p \hat{\otimes} Y_1)^* 
 \stackrel {(\id_{X_p} \otimes  \tilde{d}_{1} )^*} { \longrightarrow} 
(X_p \hat{\otimes} Y_2)^* 
 \cdots   \]
splits as a complex of Banach spaces, so that there exist bounded linear operators
\[ 
\tilde{\gamma}_{(t, X_{p},Y_n)}:
(X_p \hat{\otimes} Y_n)^* \rightarrow (X_{p} \hat{\otimes} Y_{n-1})^*
\] 
such that
\begin{eqnarray}
\label{tilde-gamma-of-t}
\nonumber \; \tilde{\gamma}_{(t, X_{p},Y_{n+1})} \circ (\id_{X_p} \otimes \tilde{d}_{n})^*  +  (\id_{X_p} \otimes \tilde{d}_{n-1})^* \circ \tilde{\gamma}_{(t, X_{p},Y_{n})} = \id_{(X_{p} \hat{\otimes} Y_n)^*},\;\\
\;\; (\id_{X_p} \otimes \varepsilon_2)^* \circ \tilde{\gamma}_{(s, X_{p},Y_0)} + \tilde{\gamma}_{(s, X_{p},Y_1)} \circ 
(\id_{X_p} \otimes \tilde{d}_{0})^*  = \id _{(X_p \hat{\otimes}  Y_0)^*} \;\; \text{\rm and} \\
\nonumber\; \tilde{\gamma}_{(s, X_{p},Y_0)} \circ ( \id_{X_p} \otimes \varepsilon_2 )^*  = \id _{(X_p \hat{\otimes} Y)^*}.
\end{eqnarray}

Consider the two morphisms 
$\id_{({\mathcal X} \hat{\otimes }{\mathcal Y})^*}$ and 
$(\varepsilon_1 \otimes \id_{\mathcal Y})^* \circ  \gamma_{(s, X_{0},{\mathcal Y})} $ of the complex $({\mathcal X} \hat{\otimes }{\mathcal Y})^*.$
Here $(\varepsilon_1 \otimes \id_{\mathcal Y})^* \circ  \gamma_{(s, X_{0},{\mathcal Y})} $ is trivial on $(X_m \ptp Y_p)^*$ for $m \ge 1$ and
equal to $(\varepsilon_1 \otimes \id_{Y_p})^* \circ  \gamma_{(s, X_{0},Y_p)} $ on $(X_0 \ptp Y_p)^*$.
By virtue of Lemma \ref{CommutativeWeakAdmis} and equalities \ref{gamma-of-s}, $\gamma_{(s, {\mathcal X},{\mathcal Y})}$ which is equal to 
$\gamma_{(s, X_{p},Y_n)}$ on $(X_p \ptp Y_n)^*$
is a homotopy between $\id_{({\mathcal X} \hat{\otimes }{\mathcal Y})^*}$ and $(\varepsilon_1 \otimes \id_{\mathcal Y})^* \circ  \gamma_{(s, X_{0},{\mathcal Y})} $, so
\[
\gamma_{(s, {\mathcal X},{\mathcal Y})}: \id_{({\mathcal X} \hat{\otimes }{\mathcal Y})^*} \cong (\varepsilon_1 \otimes \id_{\mathcal Y})^* \circ  \gamma_{(s, X_{0},{\mathcal Y})} : ({\mathcal X} \hat{\otimes }{\mathcal Y})^* \to ({\mathcal X} \hat{\otimes }{\mathcal Y})^*.
\]
Similarly, by virtue of 
Lemma \ref{CommutativeWeakAdmis} and equalities \ref{tilde-gamma-of-t},
$\tilde{\gamma}_{(t, {\mathcal X},{\mathcal Y})}$ which is equal to
$\tilde{\gamma}_{(t, X_{p},Y_n)}$ on $(X_p \ptp Y_n)^*$
is a homotopy between $\id_{({\mathcal X} \hat{\otimes }{\mathcal Y})^*}$ and $
( \id_{\mathcal X} \otimes \varepsilon_2 )^* \circ \tilde{\gamma}_{(t, {\mathcal X},Y_0)}$, so
\[
\tilde{\gamma}_{(t, {\mathcal X},{\mathcal Y})}: \id_{({\mathcal X} \hat{\otimes }{\mathcal Y})^*} \cong 
( \id_{\mathcal X} \otimes \varepsilon_2)^* \circ \tilde{\gamma}_{(t, {\mathcal X}, Y_0)}: ({\mathcal X} \hat{\otimes }{\mathcal Y})^* \to ({\mathcal X} \hat{\otimes }{\mathcal Y})^*.
\] 
Here 
$( \id_{\mathcal X} \otimes \varepsilon_2)^* \circ \tilde{\gamma}_{(t, {\mathcal X}, Y_0)}$ is trivial on $(X_m \ptp Y_q)^*$ for $q \ge 1$ and
equal to $(\id_{X_m} \otimes \varepsilon_2)^* \circ \tilde{\gamma}_{(t, {X_m}, Y_0)}$ on $(X_m \ptp Y_0)^*$.
Therefore, by \cite[Proposition II.2.3]{Mac}, there exists a product homotopy
\[
\tilde{\gamma}_{(t, {\mathcal X},{\mathcal Y})} + \gamma_{(s, {\mathcal X},{\mathcal Y})} \circ ( \id_{\mathcal X} \otimes \varepsilon_2)^* \circ \tilde{\gamma}_{(t, {\mathcal X}, Y_0)}
\]
between morphisms 
$\id_{({\mathcal X} \hat{\otimes }{\mathcal Y})^*}$
and 
$(\varepsilon_1 \otimes \id_{\mathcal Y})^* \circ  \gamma_{(s, X_{0},{\mathcal Y})} \circ ( \id_{\mathcal X} \otimes \varepsilon_2)^* \circ \tilde{\gamma}_{(t, {\mathcal X}, Y_0)}$
 of the complex $ ({\mathcal X} \hat{\otimes }{\mathcal Y})^*$.
One can check that 
$$(\varepsilon_1 \otimes \id_{\mathcal Y})^* \circ  \gamma_{(s, X_{0},{\mathcal Y})} \circ ( \id_{\mathcal X} \otimes \varepsilon_2)^* \circ \tilde{\gamma}_{(t, {\mathcal X}, Y_0)}$$
is trivial on $({\mathcal X} \hat{\otimes }{\mathcal Y})^*_n$ for $n \ge 1$ and
equal to 
$(\varepsilon_1 \otimes \varepsilon_2)^* \circ  \gamma_{(s, X_{0},Y)} \circ \tilde{\gamma}_{(t, X_0, Y_0)}$ on $(X_0 \hat{\otimes} Y_0)^*$.
Note that 
\[
\gamma_{(s, X_{0},Y)} \circ \tilde{\gamma}_{(t, X_0, Y_0)} \circ
(\varepsilon_1 \otimes \varepsilon_2)^* = \id_{(X \hat{\otimes } Y)^*}.
\]
Thus the dual complex $0 \rightarrow (X \hat{\otimes} Y)^* \stackrel{(\varepsilon_1\otimes\varepsilon_2)^*}{\longrightarrow} 
({\mathcal X} \hat{\otimes }{\mathcal Y})^*$
splits as a complex of Banach spaces.
\end{proof} 

The following lemma is  essentially  \cite[Lemma 3.6]{LW}.

\begin{lemma}\label{pseudo-resol-Abai} Let $\A$ be a Banach algebra and let  $I$  be a closed two-sided ideal of $\A_+$.
 Suppose that one of the following conditions is satisfied:

{\rm (i)} $I$ is flat in $\A$-$\mod$ and has a  left bounded approximate identity  $(e_{\alpha})_{ \alpha \in \Lambda}$; 

{\rm (ii)} $I$ is flat in $\mod$-$\A$ and has a right bounded approximate identity $(e_{\alpha})_{ \alpha \in \Lambda}$; 

{\rm (iii)} $I$ has a  bounded approximate identity 
 $(e_{\alpha})_{ \alpha \in \Lambda}$. 

Then the sequence 
\begin{equation}
\label{psres}
 0 \longleftarrow I \stackrel {\varepsilon} {\longleftarrow} 
 I {\mathbin{\widehat{\otimes}}} I 
 \stackrel {d_0} {\longleftarrow} I {\mathbin{\widehat{\otimes}}} I {\mathbin{\widehat{\otimes}}} 
I {\longleftarrow}
\dots   {\longleftarrow} I^{ {\mathbin{\widehat{\otimes}}}(n+2)}
\stackrel {d_n} {\longleftarrow}  
I^{ {\mathbin{\widehat{\otimes}}}(n+3)} \longleftarrow \dots,
\end{equation}
where $\varepsilon(b_0 \otimes b_1) = b_0 b_1$ and
\[
d_n (b_0 \otimes b_1 \otimes b_2 \otimes \dots \otimes b_{n+2} ) = 
 \sum_{i=0}^{n+1} (-1)^{i} 
( b_0 \otimes b_1 \otimes \dots \otimes b_i b_{i+1} \otimes 
\dots \otimes b_{n+2}),
\]
is a pseudo-resolution of $I$ in  $\A$-$\essmod$-$\A$ such that all modules $I^{ {\mathbin{\widehat{\otimes}}}(n)}$, $n \ge 2$, are flat in $\A$-$\mod$-$\A$.
\end{lemma}

\begin{proof} It is easy to check that $ d_{n-1} \circ d_n = 0$
for $n \ge 1$ and $\varepsilon \circ d_0 = 0$. Thus (\ref{psres}) is a
complex. By  \cite[Theorem VII.1.5]{He0}, $I$ is strictly flat as a left and as a right Banach $\A$-module. Hence, by \cite[Proposition VII.2.4]{He0}, for any $n \ge 2$, the Banach  $\A$-bimodule $I^{ {\mathbin{\widehat{\otimes}}} n}$ is flat in $\A$-$\mod$-$\A$. Note that $I^{ {\mathbin{\widehat{\otimes}}} n}$ is an essential Banach  $\A$-bimodule since  $I$ has a left or right  bounded approximate identity.

We consider the case when $I$ is flat in $\A$-$\mod$ and has a  left bounded approximate identity  $(e_{\alpha})_{ \alpha \in \Lambda}$.
Now we have to show that the dual complex 

\vspace*{0.2cm}
$ 0 \longrightarrow I^* \stackrel {\varepsilon^*} {\longrightarrow}  
(I {\mathbin{\widehat{\otimes}}} I)^* 
 \stackrel {d_0^*} {\longrightarrow} (I {\mathbin{\widehat{\otimes}}} I {\mathbin{\widehat{\otimes}}} I)^* 
{\longrightarrow}
\dots $
\begin{equation}
\label{dpsres}
\hspace*{2.0cm} \dots {\longrightarrow}  (I^{ {\mathbin{\widehat{\otimes}}}(n+2)})^* 
\stackrel {d_n^*} {\longrightarrow} (I^{ {\mathbin{\widehat{\otimes}}}(n+3)})^* 
\longrightarrow \dots
\end{equation}
is admissible. 

   Consider the Fr${\rm \acute{e}}$chet filter $F$ on $\Lambda$, with 
 base $ \{ Q_{\lambda}: 
\lambda \in \Lambda \}$, where 
$Q_{\lambda} = \{ \alpha \in \Lambda : \alpha \ge \lambda \}.$ Thus
\newline
\[
 F = 
\{ E \subset \Lambda : \;  {\rm there \; is \; a} \; \lambda  \in \Lambda \; {\rm such \; that } \; 
 Q_{\lambda} \subset E \}.
\]
\noindent Let $U$ be an ultrafilter on $ \Lambda $ which refines $F$. One can find information on 
 filters in \cite{Bo}. 

 For $n \ge -1$ and $f \in (I^{ {\mathbin{\widehat{\otimes}}}(n+3)})^*$, we define
 $g_f \in (I^{ {\mathbin{\widehat{\otimes}}}(n+2)})^* $ by
\[
g_f(u) = \lim_{ \alpha \rightarrow U}
 f(e_{\alpha} \otimes u) \; 
{\rm for \; all} \; u \in I^{ {\mathbin{\widehat{\otimes}}}(n+2)}.
\]
One can check the following: $g_f$ is a bounded linear functional, 
the operator 
\[
s_n : (I^{ {\mathbin{\widehat{\otimes}}}(n+3)})^* \rightarrow (I^{ {\mathbin{\widehat{\otimes}}}(n+2)})^* 
: f \mapsto g_f
\]
is a bounded linear operator, 
$d_{n-1}^* \circ s_{n-1} + s_{n} \circ d_{n}^* =
id _{(I^{ {\mathbin{\widehat{\otimes}}}(n+2)})^*}$ for all $n \ge 1$ and 
$\varepsilon^* \circ s_{-1} + s_{0} \circ d_{0}^* =
id _{(I^{ {\mathbin{\widehat{\otimes}}}2})^*}$.
Thus (\ref{dpsres}) is admissible  (see \cite[III.1.9]{He0}).
Therefore, by definition, (\ref{psres})  
is a pseudo-resolution of $I$ in  $\A$-$\essmod$-$\A$ such that all modules $I^{ {\mathbin{\widehat{\otimes}}}(n)}$, $n \ge 2$, are flat in $\A$-$\mod$-$\A$.
\end{proof}

\section{WEAK BIDIMENSION OF BANACH ALGEBRAS WITH BOUNDED APPROXIMATE IDENTITIES}

Let $\A$ be a Banach algebra and $X$ be a Banach
$\A$-bimodule.  Let us recall the definition of the standard homological
chain complex ${\mathcal C}_{\sim} (\A, X)$. For $n\ge 0$, let $C_n(\A, X)$
denote the projective tensor product 
 $X \ptp \A^{{\ptp} n}$. 
The elements of $C_n(\A, X)$ are called
{\em $n$-chains}. Let the differential $d_n : C_{n+1} \to C_n$ be given by
\begin{eqnarray*} 
d_n(x \otimes  a_1 \otimes \ldots \otimes  a_ {n+1})\!
&=& \! x \cdot a_1 \otimes \ldots \otimes  a_ {n+1}\\
& & \! + \sum_{k=1}^{n} (-1)^k (x \otimes a_1 \otimes \ldots \otimes a_k
a_{k+1} \otimes \ldots \otimes a_ {n+1})\\ 
& &\! +(-1)^{n+1}(a_ {n+1} \cdot x  \otimes a_1 \otimes \ldots \otimes 
a_{n})
\end{eqnarray*}
with $d_{-1}$ the null map. The space of boundaries 
$B_n({\mathcal C}_{\sim} (\A, X))= \im d_n$ is denoted by $B_n(\A, X)$ and
the space of cycles $Z_n({\mathcal C}_{\sim} (\A, X)) = \Xer d_{n-1}$ is 
denoted by $Z_n(\A, X)$.
The homology groups of  this complex
$H_n({\mathcal C}_{\sim} (\A, X)) = Z_n(\A, X)/B_n(\A, X)$ are called  the
{\it Hochschild homology groups of  $\A$ with coefficients in $X$} and
denoted by $\H_n(\A, X)$ \cite[Definition II.5.28]{He0}.

The Hochschild cohomology groups $\H^n(\A, X^*)$ of the Banach algebra $\A$
with coefficients in the dual $\A$-bimodule $X^*$ are topologically isomorphic to 
the cohomology groups $H^n(({\mathcal C}_{\sim} (\A, X))^*)$ of  the 
dual complex $({\mathcal C}_{\sim} (\A, X))^*$, see \cite{BEJ1} and
\cite[Definition I.3.2 and Proposition II.5.27]{He0}.

Let $\A$  be a Banach  algebra with a bounded approximate identity
$(e_{\alpha})_{\alpha \in \Lambda}$. We put $\beta_n(\A )=
\A^{{\ptp}^{n+2}},$ $n \ge 0$, and let 
 $d_n:\beta_{n+1}(\A) \to \beta_{n}(\A)$ be given by
$$d_n( a_0 \otimes \ldots \otimes  a_ {n+2})= 
\sum_{k=0}^{n+1} (-1)^k ( a_0 \otimes \ldots \otimes a_k a_{k+1} \otimes
\ldots \otimes a_ {n+2}).$$ By Lemma \ref{pseudo-resol-Abai}, the
complex
$$
 0 \leftarrow \A \stackrel{\pi}{\longleftarrow} \beta_{0}(\A)  
\stackrel{d_0}{\longleftarrow} \beta_{1}(\A) 
\stackrel{d_1}{\longleftarrow} \cdots 
\from \beta_n(\A )  \stackrel{d_n}{\longleftarrow} \beta_{n+1}(\A)  \from
\dots, $$ 
where $\pi:\beta_{0}(\A ) \to \A: a \otimes b \mapsto ab $,
is a flat pseudo-resolution of the $\A$-bimodule $\A$. We denote it by
$ 0 \leftarrow \A \stackrel{\pi}{\longleftarrow} \beta(\A )$.

Further we will need the following  result of the author and M.C. White. 

\begin{proposition}\label{H*Ibai-idealA} {\rm \cite[Proposition 5.2 (i)]{LW}}  Let $\A$ be a Banach algebra and 
let  $I$  be a closed two-sided ideal of $\A$.
 Suppose that $I$ has a  bounded approximate identity.
Then, for any Banach $I$-bimodule $Z$,
\[
{\mathcal H}_n(I,Z) = 
{\mathcal H}_n(\A,\overline{IZI}) \; {\rm and } \;
  {\mathcal H}^n(I,Z^*) = 
{\mathcal H}^n(\A,(\overline{IZI})^*)  \; 
{\rm for \; all} \;  n \ge 1.
\]
\end{proposition}

\begin{remark} In Theorem \ref{AptpAone-side-bai} we prove  that the homological properties of the tensor product algebras $\A {\mathbin{\widehat{\otimes}}} \B$ and $\A_+ {\mathbin{\widehat{\otimes}}} \B_+$
are different. Thus in proofs of homological properties of $\A {\mathbin{\widehat{\otimes}}} \B$ one needs to avoid adjoining an identity to the algebra. On the other hand, the previous Proposition \ref{H*Ibai-idealA} shows that
in the case of Banach algebras $\A$ with bounded approximate identities
we can restrict ourselves to the category of essential Banach modules in questions on $db_w$ and  ${\mathcal H}^n(\A,X^*)$. In the next  propositions we develop standard homological tools for injective and flat essential Banach modules without adjoining an identity to the algebra. We  present results for one of the categories: $\A$-mod, mod-$\A$ or
$\A$-mod-$\A$; for the other categories similar results hold.
\end{remark}
Let $\A$ be a Banach algebra.
Recall that, for $X \in \A$-mod, the canonical morphism 
\[ \pi_+: \A_+ {\mathbin{\widehat{\otimes}}} X \to X \]
is defined by 
\[ \pi_+(a \otimes x)= a \cdot x \;\;(a \in \A_+, x \in X).
\]
By \cite[Chapter VII]{He0}, $X^* \in \mod$-$\A$ is injective if and only if 
\[\pi_+^*:X^* \to (\A_+ {\mathbin{\widehat{\otimes}}} X)^*\]
is a coretraction in $\mod$-$\A$, that is, there is a morphism  in $\mod$-$\A$
\[ \zeta_+: (\A_+ {\mathbin{\widehat{\otimes}}} X)^*\to X^*\]
such that $\zeta_+ \circ \pi_+^*= 1_{X^*}$.

\begin{proposition}\label{ingective-ess-mod}  Let $\A$  be a
  Banach  algebra and let $X$ be a left essential Banach $\A$-module, that is, $X = \overline{\A X}$. Then $X^* \in \mod$-$\A$ is injective if and only if 
\[\pi^*:X^* \to (\A {\mathbin{\widehat{\otimes}}} X)^*\]
is a coretraction in $\mod$-$\A$, that is, there is a morphism  in   $\mod$-$\A$
\[ \zeta: (\A {\mathbin{\widehat{\otimes}}} X)^*\to X^*\]
such that $\zeta \circ \pi^*= 1_{X^*}$. 
\end{proposition}

\begin{proof} Consider the natural embedding 
\[{i}:\A {\mathbin{\widehat{\otimes}}} X \to \A_+ {\mathbin{\widehat{\otimes}}} X: a \otimes x \mapsto a \otimes x.\]
Note that $\pi = \pi_+ \circ {i}$, thus 
$\pi^* = {i}^* \circ \pi_+^*.$

($\Rightarrow$)
Suppose $X^*$ is injective in $\mod$-$\A $. Thus there exists  a morphism  in  $\mod$-$\A$
\[ \zeta_+: (\A_+ {\mathbin{\widehat{\otimes}}} X)^*\to X^*\]
such that $\zeta_+ \circ \pi_+^*= 1_{X^*}$.

If $f \in {\rm Ker}~i^*$, then 
\[[f \cdot a] (c \otimes x) = f (ac \otimes x)=[i^*(f)](ac \otimes x)=0,
\]
for all $a \in \A$, $c \in \A_+$, $x \in X$. This implies that, for all 
$f \in {\rm Ker}~i^*$ and $a \in \A$,
\[ [\zeta_+(f)] \cdot a =  \zeta_+(f \cdot a) =0. \]
Therefore, for all 
$f \in {\rm Ker}~i^*$, $\zeta_+(f)$ is zero on $X^* = \overline{X^* \A}$.
Thus there is a unique morphism of right $\A$-modules
$ \zeta: (\A {\mathbin{\widehat{\otimes}}} X)^*\to X^*$ such that 
Diagram (\ref{ess-injective}) is commutative.

\begin{equation}
\label{ess-injective}
\begin{array}{ccccccccc}
\;(\A_+ {\mathbin{\widehat{\otimes}}} X)^* & \stackrel{\zeta_+}{\longrightarrow} &   \;X^* \; \\
 \vcenter{\llap{$\scriptstyle{i^*}$}}\downarrow  &  
      \vcenter{\rlap{$\scriptstyle{\zeta}$}}\nearrow & ~ \\ 
\;(\A {\mathbin{\widehat{\otimes}}} X)^* ~ & ~ & \\  
\end{array}
\end{equation} 
One can check that $\zeta \circ \pi^*= \id_{X^*}$.

($\Leftarrow$) Suppose that there is a morphism  in $\mod$-$\A$
\[ \zeta: (\A {\mathbin{\widehat{\otimes}}} X)^*\to X^*\]
such that $\zeta \circ \pi^*= 1_{X^*}$. Put $\zeta_+ = \zeta \circ i^*$.
It is obvious that $\zeta_+$ is a  morphism of right $\A$-modules and 
$\zeta_+ \circ \pi_+^*= \zeta \circ i^* \circ \pi_+^* = \zeta \circ \pi^* = 1_{X^*}$.
\end{proof}

\begin{proposition}\label{injective-mod-essmod}
{\rm (i)}~ Let $\A$ be a Banach algebra with a  left [right] bounded approximate identity  $(e_{\alpha})_{ \alpha \in \Lambda}$ and let 
$X$ be a left [right] essential Banach $\A$-module. 
Then $X^*$ is injective in  $\mod$-$\A$ [$\A$-$\mod$] if and only if 
$X^*$  is injective in $\essmod$-$\A$ [$\A$-$\essmod$].

{\rm (ii)}~ Let $\A$ be a Banach algebra with a bounded approximate identity  $(e_{\alpha})_{ \alpha \in \Lambda}$ and let 
$X$ be an essential Banach $\A$-bimodule. 
Then $X^*$ is injective in $\A$-$\mod$-$\A$ if and only if 
$X^*$  is injective in $\A$-$\essmod$-$\A$.
\end{proposition}

\begin{proof} (i) It is obvious that the injectivity of $X^*$ in $\mod$-$\A$ implies the injectivity of $X^*$ in $\essmod$-$\A$. 

Suppose that $X^*$ is injective  in $\essmod$-$\A$.  
As in  Lemma \ref{pseudo-resol-Abai} one can define a bounded linear operator
\[ \alpha: (\A {\mathbin{\widehat{\otimes}}} X)^*\to X^*: f \mapsto g_f, \]
where $g_f \in X^* $ is given by
\[
g_f(x) = \lim_{ \alpha \rightarrow U} f(e_{\alpha} \otimes x) \; 
{\rm for \; all} \; x \in X.
\]
One can check that  $\alpha \circ \pi^*= 1_{X^*}$. Therefore,
by \cite[Proposition III.1.14(ii)]{He0}, there is a morphism of right 
Banach $\A$-modules
\[ \zeta: (\A {\mathbin{\widehat{\otimes}}} X)^*\to X^*\]
such that $\zeta \circ \pi^*= 1_{X^*}$. By Proposition \ref{ingective-ess-mod}, $X^*$ is injective  in $\mod$-$\A$.

A proof of Part (ii) is similar to the proof of Part (i).   
\end{proof}

\begin{proposition}\label{XptpYflatAptpBbimod}  Let $\A$ and $\B$ be
  Banach  algebras, let $X$ be an essential Banach $\A$-bimodule  and let $Y$ be an essential Banach $\B$-bimodule. Suppose $X$ is flat in $\A$-mod-$\A$ and $Y$ is flat in $\B$-$\mod$-$\B$. Then  $X {\mathbin{\widehat{\otimes}}} Y$ is
 flat in $\A \ptp \B$-$\mod$-$\A \ptp \B$.
\end{proposition}

\begin{proof}  
For Banach spaces $U$ and $V$,  we will use the well-known isomorphism \cite[Theorem 2.2.17]{He0}
\[
{\mathfrak B}(U, {\mathfrak B}(V, W))  \iso {\mathfrak B}(U {\mathbin{\widehat{\otimes}}} V, W): \psi \mapsto \phi,
\]
where $\phi(u \otimes v) = [\psi(u)](v)$, $u\in U$, $v\in V$.

As in \cite[Chapter VII]{He0}, $X$ is a flat left Banach $\A \ptp \A^{op}$-module with multiplication 
\[(a_1 \otimes a_2) \cdot x = a_1  \cdot x \cdot a_2\] 
and $X^* \in \mod$-$\A \ptp \A^{op}$ is injective.
By Proposition \ref{ingective-ess-mod},  since $X$ is a left
 essential Banach $\A \ptp \A^{op}$-module, 
 there is a morphism  in  $\mod$-$\A \ptp \A^{op}$
\[ \zeta_X: ((\A \ptp \A^{op}) {\mathbin{\widehat{\otimes}}} X)^*\to X^*\]
such that $\zeta_X \circ \pi_X^*= 1_{X^*}$. Here 
the canonical morphism 
\[ \pi_X: (\A \ptp \A^{op}) {\mathbin{\widehat{\otimes}}} X \to X \]
is defined by 
\[ \pi_X (u \otimes x)= u \cdot x = a_1  \cdot x \cdot a_2 \;\;(u=a_1 \otimes a_2 \in \A \ptp \A^{op}, x \in X).
\]
Therefore we can define 
\[ \tilde{\zeta}_X: {\mathfrak B}(Y,((\A \ptp \A^{op}) {\mathbin{\widehat{\otimes}}} X)^*)\to {\mathfrak B}(Y,X^*): \phi \mapsto \zeta_X \circ \phi;\]
\[ \tilde{\pi_X^*}: {\mathfrak B}(Y,X^*) \to {\mathfrak B}(Y,((\A \ptp \A^{op}) {\mathbin{\widehat{\otimes}}} X)^*): \psi \mapsto \pi_X^* \circ \psi;\]
and one can see that $\tilde{\zeta}_X \circ \tilde{\pi_X^*}= 1_{{\mathfrak B}(Y,X^*)}$.
Thus the $\A \ptp \B$-bimodule
$ (X {\mathbin{\widehat{\otimes}}} Y)^* \stackrel{i_1}{\iso} {\mathfrak B}(Y, X^*) $ is a retract of 
\[{\mathfrak B}(Y,((\A \ptp \A^{op}) {\mathbin{\widehat{\otimes}}} X)^*) \stackrel{i_2}{\iso} 
{\mathfrak B}((\A \ptp \A^{op}) {\mathbin{\widehat{\otimes}}} X, Y^*).\]

We can consider $Y$ as a flat right Banach $\B^{op} \ptp \B$-module with multiplication 
\[y \cdot (b_1 \otimes b_2)  = b_1  \cdot y \cdot b_2.\] 
By Proposition \ref{ingective-ess-mod},  since $Y$ is a right
 essential Banach $\B^{op} \ptp \B$-module and $Y^*$ is injective in  
$\B^{op} \ptp \B$-$\mod$, 
 there is a morphism  in  $\B^{op} \ptp \B$-$\mod$
\[ \zeta_Y: (Y {\mathbin{\widehat{\otimes}}} (\B^{op} \ptp \B))^*\to Y^*\]
such that $\zeta_Y \circ \pi_Y^*= \id_{Y^*}$. 
 Here 
the canonical morphism 
\[ \pi_Y: Y {\mathbin{\widehat{\otimes}}} (\B^{op} \ptp \B) \to Y \]
is defined by 
\[ \pi_Y (y \otimes u)= y \cdot u = b_1  \cdot y \cdot b_2 \;\;(u=b_1 \otimes b_2 \in \B^{op} \ptp \B, y \in Y).
\]
Thus we can define
\[ \tilde{\zeta}_Y: {\mathfrak B}((\A \ptp \A^{op}) {\mathbin{\widehat{\otimes}}} X, (Y {\mathbin{\widehat{\otimes}}} (\B^{op} \ptp \B) )^*) \to 
{\mathfrak B}((\A \ptp \A^{op}) {\mathbin{\widehat{\otimes}}} X, Y^*) : \phi \mapsto \zeta_Y \circ \phi;\]
\[ \tilde{\pi_Y^*}: {\mathfrak B}((\A \ptp \A^{op}) {\mathbin{\widehat{\otimes}}} X, Y^*)
\to {\mathfrak B}((\A \ptp \A^{op}) {\mathbin{\widehat{\otimes}}} X, (Y {\mathbin{\widehat{\otimes}}} (\B^{op} \ptp \B) )^*): \psi \mapsto \pi_Y^* \circ \psi;\]
and $\tilde{\zeta}_Y \circ \tilde{\pi_Y^*}= 1_{{\mathfrak B}((\A \ptp \A^{op}) {\mathbin{\widehat{\otimes}}} X, Y^*)}$.
Therefore 
$ (X {\mathbin{\widehat{\otimes}}} Y)^*$ is a retract of 
\[{\mathfrak B}((\A \ptp \A^{op}) {\mathbin{\widehat{\otimes}}} X, (Y {\mathbin{\widehat{\otimes}}} (\B^{op} \ptp \B) )^*) \stackrel{i_3}{\iso}  
 {\mathfrak B}((\A \ptp \B) \ptp (\A \ptp \B)^{op} \ptp (X \ptp Y), \CC).\]

One can check that 
\[ \zeta_{X \ptp Y}= i_1^{-1} \circ \tilde{\zeta_X}  \circ i_2^{-1} \circ \tilde{\zeta_Y}  \circ i_3^{-1}: ((\A \ptp \B) \ptp (\A \ptp \B)^{op} \ptp (X \ptp Y))^* \to  (X \ptp Y)^*
\]
 is a morphism  in  $(\A \ptp \B) \ptp (\A \ptp \B)^{op}$-$\mod$, that
\[ \pi^*_{X \ptp Y}= i_3 \circ \tilde{\pi^*_Y}  \circ i_2 \circ \tilde{\pi^*_X}  \circ i_1: (X \ptp Y)^* \to ((\A \ptp \B) \ptp (\A \ptp \B)^{op} \ptp (X \ptp Y))^*, 
\]
and $\zeta_{X \ptp Y} \circ \pi_{X \ptp Y}^*= 1_{(X \ptp Y)^*}$.
By Proposition \ref{ingective-ess-mod},
$ (X {\mathbin{\widehat{\otimes}}} Y)^*$ is injective in 
$\A \ptp \B$-mod-$\A \ptp \B$ and thus $X {\mathbin{\widehat{\otimes}}} Y$ is
 flat in  $\A \ptp \B$-mod-$\A \ptp \B$.
\end{proof}

\begin{proposition}\label{ProductResolutions}  Let $\A_1$ and $\A_2$ be
  Banach  algebras.  
Let~  
 $ 0 \leftarrow X \stackrel{\varepsilon_1}{\longleftarrow}{\mathcal X}$ 
be a  pseudo-resolution of $X$ in  $\A_1$-$\essmod$-$\A_1$ such that all modules in $\X$ are flat in $\A_1$-$\mod$-$\A_1$   and
 $ 0 \leftarrow Y \stackrel{\varepsilon_2}{\longleftarrow} {\mathcal Y}$ 
be a  pseudo-resolution  of $Y$ in $\A_2$-$\essmod$-$\A_2$
such that all modules in $\Y$ are flat in $\A_2$-$\mod$-$\A_2$. Then 
$ 0 \leftarrow X {\mathbin{\widehat{\otimes}}} Y \stackrel
{\varepsilon_1\otimes\varepsilon_2}{\longleftarrow} 
{\mathcal X} {\mathbin{\widehat{\otimes }}}{\mathcal Y}$
 is a  pseudo-resolution  of
$X {\mathbin{\widehat{\otimes}}} Y$ in $\A_1 {\mathbin{\widehat{\otimes}}} \A_2$-$\essmod$-$\A_1 {\mathbin{\widehat{\otimes}}} \A_2$
such that all modules in ${\mathcal X}{\mathbin{\widehat{\otimes }}}{\mathcal Y}$ are flat in $\A_1 {\mathbin{\widehat{\otimes}}} \A_2$-$\mod$-$\A_1 {\mathbin{\widehat{\otimes}}} \A_2$. \end{proposition}

\begin{proof}  
By Definition \ref{XptpY}, 
for any $n \ge 0$, the Banach $\A_1\ptp\A_2$-bimodule 
$$(\mathcal X {\mathbin{\widehat{\otimes}}}{\mathcal Y})_n=  \displaystyle{\bigoplus_{m+q=n}} X_m \ptp Y_q.$$  
By  Proposition \ref{XptpYflatAptpBbimod}, $(\mathcal X {\mathbin{\widehat{\otimes}}}{\mathcal Y})_n$ is flat in $\A_1 {\mathbin{\widehat{\otimes}}} \A_2$-$\mod$-$\A_1 {\mathbin{\widehat{\otimes}}} \A_2$ for all $n \ge 0$.

By assumption the complexes $ 0 \leftarrow X \stackrel{\varepsilon_1}{\longleftarrow}{\mathcal X}$ and $ 0 \leftarrow Y \stackrel{\varepsilon_2}{\longleftarrow} {\mathcal Y}$ are weakly admissible.
By Proposition \ref{tensor-product-w-a-resol},
the complex ${0 \leftarrow X \hat{\otimes} Y \stackrel
{\varepsilon_1\otimes\varepsilon_2}{\longleftarrow} 
{\mathcal X} \hat{\otimes }{\mathcal Y}}$
is weakly admissible too. 
\end{proof} 

By \cite[Theorem III.4.9]{He0}, for a Banach algebra $\A$ and a Banach $\A$-bimodule $X$,
\[
  {\mathcal H}^n(\A,X) = {\rm Ext}_{\A^e}^n(\A_+, X).
\]
For a Banach algebra with a bounded approximate identity and for dual bimodules, we may avoid adjoining identity to the algebra.

\begin{proposition}\label{H-Ext-Abai} {\rm \cite[Proposition 3.12]{LW}}  Let $\A$ be a Banach algebra with a bounded approximate identity and let $X$ be an essential  Banach $\A$-bimodule.
Then, for all $n \ge 0$,
\[
  {\mathcal H}^n(\A,X^*) = {\rm Ext}_{\A^e}^n(\A, X^*).
\]
\end{proposition}

\begin{theorem}\label{dbwA<=n} Let $\A$ be a Banach algebra with bounded approximate identity.
For each integer $n \ge 0$ the following properties of a Banach algebra $\A$ are equivalent:

{\rm (i)} $db_w \A \le n$;

{\rm (ii)} ${\mathcal H}^{n+1} (\A, X^*) = \{ 0 \} \; {\rm for \; all \;} X \in \A\hbox{-} \essmod \hbox{-}\A;$

{\rm (iii)} ${\mathcal H}^{m} (\A, X^*) = \{ 0 \} \; {\rm for \; all \;} m \ge n+1 \;{\rm and \; for \; all \;} X \in \A\hbox{-} \essmod \hbox{-}\A;$

{\rm (iv)}  ${\mathcal H}_{n+1} (\A, X) = \{ 0 \}$ and  ${\mathcal H}_{n} (\A, X) $ is a Hausdorff space for all
\newline $ X \in \A\hbox{-} \essmod \hbox{-}\A;$

{\rm (v)} if 
\vspace*{0.2cm}
\hspace{0.2cm}
$0  \longleftarrow \A
\stackrel {\varepsilon} { \longleftarrow}  P_0 
 \stackrel {\phi_0} { \longleftarrow} P_1 \stackrel {\phi_1} 
{ \longleftarrow}  \cdots P_{n-1} \stackrel {\phi_{n-1}} { \longleftarrow} Y \longleftarrow 0
\hfill {(0 \leftarrow \A \leftarrow \P)} $ 
\vspace*{0.2cm}
\newline is  a pseudo-resolution  of  $\A$  in which all the modules  $P_i$ are  flat in $\A$-$\essmod$-$\A$, then $Y$ is also flat in  $\A$-$\essmod$-$\A$. 

{\rm (vi)} the $\A$-bimodule $\A$ has a flat pseudo-resolution of length $n$  in the category of $\A$-$\essmod$-$\A$.

\end{theorem}

\begin{proof} By definition, (i) $\Rightarrow$ (ii). By
Proposition \ref{H*Ibai-idealA}, for a Banach algebra $\A$ with  a  bounded approximate identity and for any Banach $\A$-bimodule $Z$,
\[
  {\mathcal H}^n(\A,Z^*) = 
{\mathcal H}^n(\A,(\overline{\A Z \A})^*)  \; 
{\rm for \; all} \;  n \ge 1.
\]
Thus (ii) $\Rightarrow$ (i) and therefore (ii) $\iff $ (i).

By \cite[Corollary 1.3 and 1.a Reduction of dimension]{BEJ1}, (ii) $\iff$ (iii) and (ii) $\iff$ (iv). 

(vi) $\Rightarrow$ (ii)~
By Proposition \ref{H-Ext-Abai}, for a Banach algebra  $\A$ with a bounded approximate identity, for  an essential $X$ Banach $\A$-bimodule
and for all $n \ge 0$,
\[
  {\mathcal H}^n(\A,X^*) = {\rm Ext}_{\A^e}^n(\A, X^*).
\]
By \cite[VII.1.19]{He0}, ${\rm Ext}_{\A^e}^n(\A, X^*)$ is the cohomology of the complex $h_{\A^e}(\P, X^*)$ where $(0 \leftarrow \A \leftarrow \P)$ is a pseudo-resolution of $\A$ in $\A$-mod-$\A$. The rest is clear.

(ii) $\Rightarrow$ (v)~ By assumption, the complex 
 $(0 \leftarrow \A \leftarrow \P)$ is weakly admissible, that is,
the dual complex  

\vspace*{0.2cm}
\hspace{0.2cm}
$0  \longrightarrow \A^*
\stackrel {\varepsilon^*} { \longrightarrow}  P_0^* 
 \stackrel {\phi_0^*} { \longrightarrow} P_1 \stackrel {\phi_1^*} 
{ \longrightarrow} P_2^* \longrightarrow  \cdots P_{n-1}^* \stackrel {\phi_{n-1}^*} { \longrightarrow} Y^* \longrightarrow 0$ 
\vspace*{0.2cm}
\newline splits as a complex of Banach spaces.
Therefore there are the following admissible short exact sequences:
\[
0  \longrightarrow \A^*
\stackrel {\varepsilon^*} { \longrightarrow}  P_0^* 
 \stackrel {\phi_0^*} { \longrightarrow} {\rm Im }~ \phi_0^* \longrightarrow 0,
\]
\[ 0\longrightarrow {\rm Im }~ \phi_{k-1}^* = {\rm Ker}~\phi_{k}^*
\stackrel {i_{k}} { \longrightarrow} P_{k}^* \stackrel {\phi_{k}^*} { \longrightarrow}  {\rm Im }~ \phi_{k}^* \longrightarrow 0,
\]
$k= 1, 2, \dots, n-2,$
and 
\[ 0\longrightarrow {\rm Im }~ \phi_{n-2}^* = {\rm Ker}~\phi_{n-1}^*
\stackrel {i_{n-1}} { \longrightarrow} P_{n-1}^* \stackrel {\phi_{n-1}^*} { \longrightarrow} Y^* \longrightarrow 0,
\]
where $i_k, k=1, \dots, n-1,$ are natural inclusions.
Thus, by \cite[Theorem III.4.4]{He0}, for every $X \in \A$-$\essmod$-$\A$, there are long exact sequences of ${\rm Ext}_{\A^e}(X, \cdot)$ associated with these admissible short exact sequences.
By assumption, all the modules $P_i$ are  flat in $\A$-$\essmod$-$\A$ and therefore 
\[{\rm Ext}_{\A^e}^n(X,P_i^*) = \{ 0 \}\] for all $n \ge 1$ and all $X \in \A$-$\essmod$-$\A$.
Hence, for every $X \in \A$-$\essmod$-$\A$,
\[ {\rm Ext}_{\A^e}^1(X, Y^*) \iso {\rm Ext}_{\A^e}^2(X, {\rm Im}~\phi_{n-2}^*) \iso {\rm Ext}_{\A^e}^3 (X,{\rm Im}~\phi_{n-3}^* ) \iso \dots
\]
\[ \iso {\rm Ext}_{\A^e}^n(X,{\rm Im}~\phi_{0}^*) \iso {\rm Ext}_{\A^e}^{n+1}(X, \A^*).\]
By \cite[Proposition III.4.13]{He0},
\[{\rm Ext}_{\A^e}^{n+1}(X, \A^*) \iso
{\rm Ext}_{\A^e}^{n+1}(\A, X^*) 
\]
and, by Proposition \ref{H-Ext-Abai}, for the Banach algebra with bounded approximate identity,
\[
 {\rm Ext}_{\A^e}^{n+1}(\A, X^*) \iso {\mathcal H}^{n+1}(A,X^*).
\]
Therefore $Y^*$ is injective and hence $Y$ is flat in  $\A$-$\essmod$-$\A$.

It is obvious that (v) $\Rightarrow$ (vi).
\end{proof}

\begin{remark} To get the full  picture
for an arbitrary Banach algebra $\A$, one can see also \cite[Theorem 1]{Sel96} on equivalent conditions
to $db_w \A \le n$ for an integer $n \ge 0.$
\end{remark}

\begin{remark} It is clear that for   Banach  algebras with  bounded approximate identities  $\A$ and $\B$,
by  Theorem \ref{dbwA<=n} and Proposition \ref{ProductResolutions},
$ db_w (\A \ptp \B) \le  db_w \A + db_w \B.$
To get the equality here we need the following lemma of Yu. Selivanov and extensions of his Propositions 4.6.2 and 4.6.5 \cite{Sel02} to the case of Banach  algebras with  bounded approximate identities.
\end{remark}

Recall that a continuous linear operator $T: X \to Y$ between Banach spaces $X$ and $Y$
is topologically injective if it is injective and its image is closed, that is, 
$ T: X \to {\rm Im}~ T$ is a topological isomorphism. 

\begin{lemma}\label{not-injective-oper} {\rm \cite[Lemma 1]{Sel96}}
Let $E_0, E, F_0$ and  $F$ be Banach spaces, and let $S:E_0 \to E$ and 
$T:F_0 \to F$ be continuous linear operators. Suppose $S$ and $T$ are not  topologically injective. Then the continuous linear operator
\[ \Delta: E_0 \ptp F_0 \to  (E_0 \ptp F) \oplus
 (E \ptp F_0)\]
defined by 
\[ \Delta(x \otimes y)= (x \otimes T(y), S(x) \otimes y)\;\;(x \in 
E_0, y \in F_0).
\]
is not  topologically injective.
\end{lemma}


\begin{proposition}\label{dbwA<n} Let $\A$ be a Banach algebra with bounded approximate identity and let 
\begin{equation}
\label{pseudo-res-ofA}
\hspace{0.2cm}
0  \longleftarrow \A
\stackrel {\varepsilon} { \longleftarrow}  P_0 
 \stackrel {\phi_0} { \longleftarrow} P_1 \stackrel {\phi_1} 
{ \longleftarrow}   \cdots P_{n-1} \stackrel {\phi_{n-1}} { \longleftarrow} P_n \longleftarrow 0 \hspace{0.5cm}
\hfill {(0 \leftarrow \A \leftarrow \P)}  
\end{equation}
be  a flat pseudo-resolution  of  $\A$  in  $\A$-$\essmod$-$\A$. 
Then $db_w \A < n$ if and only if, for every  $X$ in  $\A$-$\essmod$-$\A$, the operator
\[
\phi_{n-1} \otimes_{\A-\A} \id_X: P_n \ptp_{\A-\A} X \to P_{n-1} \ptp_{\A-\A} X
\]
is topologically injective.
\end{proposition}

\begin{proof} By Theorem \ref{dbwA<=n},  $db_w \A < n$ if and only if ${\mathcal H}^n(\A,X^*) = \{0 \}$ for every  $X$ in  $\A$-$\essmod$-$\A.$
By Proposition \ref{H-Ext-Abai}, for a Banach algebra  $\A$ with a bounded approximate identity, for  an essential $X$ Banach $\A$-bimodule
and for all $n \ge 0$,
\[
  {\mathcal H}^n(\A,X^*) = {\rm Ext}_{\A^e}^n(\A, X^*).
\]

By assumption, $\A$ is a Banach algebra with bounded approximate identity. Thus, 
by Proposition \ref{injective-mod-essmod}(ii) and \cite[Theorem VII.1.14]{He0}, $ 0 \leftarrow \A  \leftarrow ({\P},\phi) $ is a flat pseudo-resolution of $\A$  in  $\A$-$\mod$-$\A$. By \cite[Exercise VII.1.19]{He0}, 
${\rm Ext}_{\A^e}^n(\A, X^*)$ is the cohomology of the complex 
 $ {_{\A} h_{\A}} ({\P}, X^*)$. Therefore, up to topological isomorphism, ${\rm Ext}_{\A^e}^n(\A, X^*)$ is the $n$-th cohomology of 
the complex ${_{\A} h_{\A}} (\P, X^*):$

\begin{equation}
\label{complexAhA(P,X*)}
0  \longrightarrow  {_{\A} h_{\A}}(P_0, X^*)
 \stackrel {h(\phi_0)} { \longrightarrow}  \cdots   \longrightarrow
{_{\A} h_{\A}}(P_{n-1}, X^*) \stackrel {h(\phi_{n-1})} { \longrightarrow} {_{\A} h_{\A}}(P_n , X^*) \longrightarrow 0 , 
\end{equation}
where $h(\phi_{n-1})$ is the operator defined by 
$h(\phi_{n-1})(\eta) = \eta \circ \phi_{n-1}$ for 
\newline $ \eta \in {_{\A} h_{\A}}(P_{n-1}, X^*)$.
Therefore
${\mathcal H}^n(\A,X^*) = \{0 \}$ for every  $X$ in  $\A$-$\essmod$-$\A$ if and only if $h(\phi_{n-1})$ is surjective for every  $X$ in  $\A$-$\essmod$-$\A$.

By the conjugate associativity law \cite[Theorem  II.5.21]{He0},
there is a natural isomorphism of Banach spaces:
\begin{eqnarray*}
{_{\A} h_{\A}} (\P, X^*)&\cong & (\P {\mathbin{\widehat{\otimes}}}_{\A^e} X)^*.
\end{eqnarray*}
It is clear that $h(\phi_{n-1})$ is the dual to $\phi_{n-1} \otimes_{\A-\A} \id_X$. By \cite[Corollary 8.6.15]{Ed}, $h(\phi_{n-1})$ is surjective if and only if $\phi_{n-1} \otimes_{\A-\A} \id_X$ is 
 topologically injective.
\end{proof}


\begin{proposition}\label{dbwAptpB>=max(dbwA,dbwB)}
Let $\A$ and $\B$  be  Banach  algebras with  bounded approximate 
identities. Then
\[
db_w (\A \ptp \B) \ge \max \{ db_w \A, db_w \B \}.
\]
\end{proposition}

\begin{proof} By assumption, $\A$ and $\B$ are  Banach algebras with bounded approximate identities. Thus $\A \ptp \B$ has a bounded approximate  identity too \cite{Da}. 
By Proposition \ref{injective-mod-essmod}(ii) and \cite[Theorem VII.1.14]{He0}, for a Banach algebra $\D$ with a bounded approximate identity, $X$ is flat in $\D$-$\mod$-$\D$ if and only if $X$ is flat in $\D$-$\essmod$-$\D$.

 Suppose that $db_w (\A \ptp \B) =n < \infty$. 
By Theorem \ref{dbwA<=n},
the $\A \ptp \B$-bimodule $\A \ptp \B$ has a flat pseudo-resolution 
$ 0 \leftarrow \A\ptp \B  \leftarrow {\P}$
of length $n$  in the category $\A \ptp \B$-$\essmod$-$\A \ptp \B$.
By \cite[Proposition VII.2.2]{He0}, $ 0 \leftarrow \A\ptp \B  \leftarrow {\P}$ is  a flat pseudo-resolution  in the category $\A $-$\essmod$-$\A$. Therefore, for every essential Banach $\A$-bimodule $X$, 
${\rm Ext}_{\A^e}^{n+1}(\A \ptp \B, X^*) =\{0 \}.$

It is easy to see that $\A$ is isomorphic to $\A \ptp \CC$ in 
$\A$-$\mod$-$\A$. On the other hand $\A \ptp \CC$ is a direct module summand of $\A \ptp \B \in \A$-$\mod$-$\A$. Therefore, since 
${\rm Ext}$ is additive, for every essential  Banach $\A$-bimodule $X$,
\[
  {\mathcal H}^{n+1}(\A,X^*) = {\rm Ext}_{\A^e}^{n+1}(\A, X^*) =\{0 \}.
\]
Thus, by Theorem \ref{dbwA<=n}, $db_w \A \le n$ and $db_w (\A \ptp \B) \ge  db_w \A.$ As in  the case of $\A$ one can show that $db_w (\A \ptp \B) \ge  db_w \B.$ 
\end{proof}

\begin{theorem}\label{dbwAptpB}
Let $\A$ and $\B$  be  Banach  algebras with  bounded approximate 
identities. Then
\[
db_w (\A \ptp \B) =  db_w \A + db_w \B.
\]
\end{theorem}

\begin{proof} By  Theorem \ref{dbwA<=n} and Proposition \ref{ProductResolutions},
\[
db_w (\A \ptp \B) \le  db_w \A + db_w \B.
\]
Therefore, by Proposition \ref{dbwAptpB>=max(dbwA,dbwB)},
\[
 \max \{ db_w \A, db_w \B \} \le db_w (\A \ptp \B) \le  db_w \A + db_w \B.
\]
Hence, in the case $db_w \A$ [$db_w \B$] is equal to $0$ or $\infty$,
$ db_w (\A \ptp \B) =  db_w \A + db_w \B.$

Suppose $db_w \A = m$ and $db_w \B = q$ where $0 < m,q < \infty$.
By Theorem \ref{dbwA<=n}, there is a flat pseudo-resolution 
$0 \leftarrow \A \stackrel{\varepsilon_1}{\longleftarrow} (\P, \phi) $
of length $m$  in the category  $\A$-$\essmod$-$\A$. By Proposition \ref{dbwA<n}, there exists $X \in \A$-$\essmod$-$\A$ such that  the operator
\[
\phi_{m-1} \otimes_{\A-\A} \id_X: P_m \ptp_{\A-\A} X \to P_{m-1} \ptp_{\A-\A} X
\]
is not topologically injective. Similarly, $db_w \B = q$ implies that
 there is
a flat pseudo-resolution 
$0 \leftarrow \B \stackrel{\varepsilon_2}{\longleftarrow} (\Q, \psi) $
of length $q$  in the category  $\B$-$\essmod$-$\B$ and 
there exist $Y \in \B$-$\essmod$-$\B$ such that  the operator
\[
\psi_{q-1} \otimes_{\B-\B} \id_Y: Q_q \ptp_{\B-\B} Y \to Q_{q-1} \ptp_{\B-\B} Y
\]
is not topologically injective.

By Proposition \ref{ProductResolutions},
$ 0 \leftarrow \A {\mathbin{\widehat{\otimes}}} \B \stackrel
{\varepsilon_1\otimes\varepsilon_2}{\longleftarrow} 
({\mathcal \P} {\mathbin{\widehat{\otimes}}} {\mathcal \Q, \delta)}$
 is a  flat pseudo-resolution  of
$\A \ptp \B$ in $\A {\mathbin{\widehat{\otimes}}} \B$-$\essmod$-$\A {\mathbin{\widehat{\otimes}}} \B$ of length $m+q$.
Take $Z = X \ptp Y$ in $\A {\mathbin{\widehat{\otimes}}} \B$-$\essmod$-$\A {\mathbin{\widehat{\otimes}}} \B$.
By Definition \ref{XptpY},  $(\P \ptp \Q)_{m+q}= P_m \ptp \Q_{q}$, 
$(\P \ptp \Q)_{m+q-1}=(P_{m-1} \ptp \Q_{q}) \oplus  (P_m \ptp \Q_{q-1})$
and
\[\delta_{m+q-1}:(\P \ptp \Q)_{m+q} \to (\P \ptp \Q)_{m+q-1}
\]
is defined by
$$ \delta_{m+q-1} (x \otimes y) = \phi_{m-1}(x) \otimes y + (-1)^m x \otimes   \psi_{q-1}(y),\;\; x \in P_m, y \in Q_q. $$
By Lemma \ref{not-injective-oper},  the operator
\[
\Delta: (\P_m \ptp_{\A-\A} X) \ptp (\Q_{q} \ptp_{\B-\B} Y) \to
\]
\[
\left((\P_{m-1} \ptp_{\A-\A} X) \ptp (\Q_{q} \ptp_{\B-\B} Y) \right) \oplus
\left((\P_m \ptp_{\A-\A} X) \ptp (\Q_{q-1} \ptp_{\B-\B} Y)\right)
\]
which is defined by 
\[\Delta(u \otimes v) = (\phi_{m-1} \otimes_{\A-\A} \id_X )(u) \otimes v +  u \otimes ((-1)^m \psi_{q-1} \otimes_{\B-\B} \id_Y )(v)
\]
is not topologically injective.
Note that there are natural isometric isomorphisms of Banach spaces
\[
(\P \ptp \Q)_{m+q} \ptp_{ \A \ptp \B-\A \ptp \B} Z \stackrel
{i_1}{\cong}
(\P_m \ptp_{\A-\A} X) \ptp (\Q_{q} \ptp_{\B-\B} Y)
\]
and 
\[
(\P \ptp \Q)_{m+q-1} \ptp_{ \A \ptp \B-\A \ptp \B} Z \stackrel
{i_2}{\cong}
\]
\[
\left((\P_{m-1} \ptp_{\A-\A} X) \ptp (\Q_{q} \ptp_{\B-\B} Y) \right) \oplus
\left((\P_m \ptp_{\A-\A} X) \ptp (\Q_{q-1} \ptp_{\B-\B} Y)\right).
\]
Thus one can see that the operator
\[
\delta_{m +q-1} \otimes_{\A \ptp \B-\A \ptp \B} \id_{Z}:
\]
\[
(\P \ptp \Q)_{m+q} \ptp_{ \A \ptp \B-\A \ptp \B} Z \to
(\P \ptp \Q)_{m+q-1} \ptp_{ \A \ptp \B-\A \ptp \B} Z
\]
which is equal to the operator $i^{-1}_2 \circ \Delta \circ i_1$ is not topologically injective.
Therefore, by Proposition \ref{dbwA<n}, $db_w (\A \ptp \B) = m +q.$
\end{proof}

\begin{theorem}\label{AptpAone-side-bai}
Let $\A$ and $\B$  be biflat Banach  algebras. Then\\
{\rm (i)}
\[
 \; \; db_w (\A \ptp \B) = 0  \;\; {\rm and} \;\; db_w (\A_+ \ptp \B_+) = db_w \A + db_w \B= 0 
\]
\hspace{1cm} if $ \A$ and $\B$ have bounded approximate  identities (b.a.i.);\\
{\rm (ii)}
\[
 \; \;db_w (\A \ptp \B) \le 1  \;\; {\rm and} \;\; db_w (\A_+ \ptp \B_+) = db_w \A + db_w \B = 2\] 
\hspace{1cm} if $ \A$ and $\B$  have  left [right], but 
 not two-sided b.a.i.;\\
{\rm (iii)} 
\[
\; \;db_w (\A \ptp \B) \le 2 \;\; {\rm and} \;\; db_w (\A_+ \ptp \B_+)= db_w \A + db_w \B = 4\] 
\hspace{1cm} if $ \A$ and $\B$  have neither  left  nor right  b.a.i.
\end{theorem}

\begin{proof} By  \cite[Theorem 6]{Sel95}, for a biflat
 Banach algebra $\D$,
\[
db_w \D = 
\begin{cases}
0 & \text{if $ \D$  has a b.a.i.,}
\\
1 & \text{if $ \D$  has  a left or right, but 
 not two-sided b.a.i.,}
\\
2 & \text{if $ \D$  has  neither a  left  nor  a
right  b.a.i.}
\end{cases}
\]
It is known that for any Banach algebra  $\D$, $db_w \D = db_w \D_+$. Hence, by \cite[Theorem 4.6.8]{Sel02}, for unital Banach algebras $\A_+$ and $\B_+$,
\[
db_w (\A_+ \ptp \B_+ )=  db_w \A_+ + db_w \B_+ = db_w \A + db_w \B.
\]
By \cite[Proposition VII.2.6]{He0}, a biflat Banach algebra is essential. Hence, by Proposition \ref{XptpYflatAptpBbimod}, the tensor product Banach algebra $\A \ptp \B$ is biflat. Note that $\A \ptp \B$ has a [left] (right) bounded approximate  identity if $\A$ and $\B$ have
 [left] (right) bounded approximate  identities \cite{Da}. 
The rest is clear.
\end{proof}

\begin{example} 
In \cite[Theorem 5.3.2]{Sel02} Yu. Selivanov proved that the algebra $\K(\ell_2 \ptp \ell_2)$ of compact operators on $ \ell_2 \ptp \ell_2$ is a  biflat Banach  algebra with  a left, but  not two-sided bounded approximate identity. Therefore, by Theorem \ref{AptpAone-side-bai},
for an integer $n \ge 1$,
\[
db_w [\K(\ell_2 \ptp \ell_2)]^{\ptp n} \le 1 \; \;{\rm and} \;\;
db_w [\K(\ell_2 \ptp \ell_2)_+]^{\ptp n}=n.
\]
\end{example}
\begin{example}
 Let $(E,F)$ be a pair of infinite-dimensional Banach spaces endowed with a nondegenerate jointly continuous bilinear 
form $ \langle \cdot, \cdot \rangle : E \times F \to {\mathbb C}$
that is not identically zero. The space ${\mathcal A}=E {\mathbin{\widehat{\otimes}}} F $ is a ${\mathbin{\widehat{\otimes}}}$-algebra with respect to the multiplication defined by 
$$ (x_1 \otimes x_2) (y_1 \otimes y_2) = \langle x_2, y_1 \rangle  
x_1 \otimes y_2, \; x_i \in E,\; y_i \in F.$$
Then ${\mathcal A}=E {\mathbin{\widehat{\otimes}}} F $ is called the {\it tensor algebra generated by the duality} 
$(E, F, \langle \cdot, \cdot \rangle)$.
In \cite[Examples 2.1.13 and 2.1.14]{Sel02}
Yu.V. Selivanov proved that $\A$ is biprojective, and,
by \cite[Theorem 2.6.21 and Corollary 2.6.24]{Sel02}, has  neither a  left  nor  a right bounded approximate identity.

In particular, if $E$ is a Banach space with the approximation
property, then the algebra ${\mathcal A}=E {\mathbin{\widehat{\otimes}}} E^* $ is isomorphic to
the algebra ${\mathcal N}(E)$ of nuclear operators on $E$
\cite[II.2.5]{He0}. 

Therefore, by Theorem \ref{AptpAone-side-bai}, for an integer $n \ge 1$,
\[
db_w [E {\mathbin{\widehat{\otimes}}} F ]^{\ptp n} \le 2 \; \;{\rm and} \;\;
db_w [(E {\mathbin{\widehat{\otimes}}} F)_+]^{\ptp n}=2n.
\]
\end{example}
\begin{example}
 Let $\B$  be the algebra of $2 \times 2$-complex matrices of the form 
\[
\left[ \begin{array}{cc}
a & b
\\
0 & 0
\end{array} \right]
\]
with matrix multiplication and norm. It is known that $\B$ is $2$-amenable, biprojective, has a left, but not right identity  \cite{Pat2}.
Therefore, by Theorem \ref{AptpAone-side-bai}, for an integer $n \ge 1$,
\[
db_w [\B]^{\ptp n} = 1, \;\;{\rm and} \;\;db_w [\B_+]^{\ptp n}=n; \;
\]
\[
db_w [\B \ptp \K(\ell_2 \ptp \ell_2)]^{\ptp n} = 1, 
 \;\;{\rm and} \;\;
db_w [\B_+ \ptp \K(\ell_2 \ptp \ell_2)_+]^{\ptp n}=2n. 
\]
\end{example}

\section{EXTERNAL PRODUCTS OF HOCHSCHILD COHOMOLOGY OF BANACH ALGEBRAS WITH BOUNDED APPROXIMATE IDENTITIES}

\begin{theorem}\label{ExternalProductExt}
Let $\A$ and $\B$  be  Banach  algebras with  bounded approximate 
identities, let $X$  be an essential  Banach 
 $\A$-bimodule and let $Y$ be an essential  Banach  $\B$-bimodule.
 Then for $n \ge 0$, up to topological
isomorphism, 
\[ \H^n(\A \ptp \B, (X \ptp Y)^*) = 
H^n(({\mathcal C}_{\sim}(\A, X) \ptp {\mathcal C}_{\sim}(\B, Y))^*). \]
\end{theorem}

\begin{proof}  Consider the flat
 pseudo-resolutions 
$0 \leftarrow \A \stackrel {\pi_{\A}} {\longleftarrow} \beta(\A )$
and 
$0 \leftarrow \B \stackrel {\pi_{\B}} {\longleftarrow} \beta(\B )$ of 
 $\A$ and $\B$ in the categories of bimodules. By
Proposition \ref{ProductResolutions} their projective tensor product
$\beta(\A ){\mathbin{\widehat{\otimes}}} \beta(\B )$ is an $\A\ptp\B$-flat  
pseudo-resolution
of $\A{\mathbin{\widehat{\otimes}}}\B$ in $\A \ptp \B$-$\mod$-$\A \ptp \B$.

By \cite[Proposition 3.7]{LW},
since the Banach algebra $\A \ptp \B$ has a bounded approximate identity,
 for $n \ge 0$, up to topological
isomorphism, 
\[ \H^n(\A \ptp \B, (X \ptp Y)^*) = 
\ext^n_{(\A \ptp \B)^e}(\A \ptp \B, (X \ptp Y)^*).\]
By \cite[Exercise VII.1.19]{He0}, 
$\ext^n_{(\A \ptp \B)^e}(\A \ptp \B, (X \ptp Y)^*)$
is the cohomology of the complex 
 $ {_{\A \ptp \B} h_{\A \ptp \B}} ({\F}, (X \ptp Y)^*)$, where 
$ 0 \leftarrow \A \ptp \B \leftarrow {\F}$ is a flat pseudo-resolution of 
$\A \ptp \B$
in  $\A \ptp \B$-$\mod$-$\A \ptp \B$. Thus,
 $\ext^n_{(\A \ptp \B)^e} (\A \ptp \B, (X \ptp Y)^*)$ can be computed 
by use of the flat
pseudo-resolution $\beta(\A ){\mathbin{\widehat{\otimes}}} \beta(\B )$.
Hence 
\newline $\ext^n_{(\A \ptp \B)^e} (\A \ptp \B, (X \ptp Y)^*)$
is the cohomology of the complex 
\newline $ {_{\A \ptp \B} h_{\A \ptp \B}} 
(\beta(\A ){\mathbin{\widehat{\otimes}}} \beta(\B ), (X \ptp Y)^*).$
By the conjugate associativity law \cite[Theorem  II.5.21]{He0},
there is a natural isomorphism of Banach spaces:
\begin{eqnarray*}
{_{\A \ptp \B} h_{\A \ptp \B}} 
(\beta(\A ){\mathbin{\widehat{\otimes}}} \beta(\B ), (X \ptp Y)^*)&\cong & 
((\beta(\A ){\mathbin{\widehat{\otimes}}}\beta(\B ))
{\mathbin{\widehat{\otimes}}}_{(\A \ptp \B)^e}(X \ptp Y))^*\\ 
 &\cong & 
((X \ptp Y) {\mathbin{\widehat{\otimes}}}_{(\A \ptp \B)^e}
(\beta(\A ){\mathbin{\widehat{\otimes}}}\beta(\B )))^*.
\end{eqnarray*}

By \cite[Proposition II.3.12]{He0}, one can see that 
the following chain complexes are topologically isomorphic:
\begin{eqnarray*} 
(X \ptp Y) {\mathbin{\widehat{\otimes}}}_{(\A \ptp \B)^e}
(\beta(\A ){\mathbin{\widehat{\otimes}}}\beta(\B )) &\cong & 
(X {\mathbin{\widehat{\otimes}}}_{\A^e} \beta(\A )){\mathbin{\widehat{\otimes}}}
(Y {\mathbin{\widehat{\otimes}}}_{\B^e} \beta(\B ))\\
&\cong & 
{\mathcal C}_{\sim}(\A,  X) \ptp {\mathcal C}_{\sim}(\B, Y).
\end{eqnarray*} 
Thus, up to topological isomorphism, 
\begin{eqnarray*} 
\H^n(\A \ptp \B, (X \ptp Y)^*) &=& 
 H^n (((X \ptp Y) {\mathbin{\widehat{\otimes}}}_{(\A \ptp \B)^e}
(\beta(\A ){\mathbin{\widehat{\otimes}}}\beta(\B )))^*)\\
 &=& 
H^n (({\mathcal C}_{\sim}(\A, X) \ptp {\mathcal C}_{\sim}(\B, Y))^*).
\qedhere
\end{eqnarray*} 
\end{proof}

Corollary \ref{ExternalProductCohom} to
Theorem \ref{ExternalProductExt} gives another proof of the K$\ddot{\rm u}$nneth formula for Hochschild cohomology groups of Banach algebras with  bounded approximate identities
(see \cite[Theorem 5.5]{GLW2}).

\begin{corollary}\label{ExternalProductCohom}
Let $\A$ and $\B$  be
  Banach  algebras with  bounded approximate identities, let $X$  
be an essential  Banach 
 $\A$-bimodule and let $Y$ be an essential  Banach  $\B$-bimodule. 
 Suppose that
all boundary maps of the standard homology complexes 
${\mathcal C}_{\sim}(\A, X)$
and ${\mathcal C}_{\sim}(\B, Y)$ have closed range and, for all n,  
$B_n(\A, X)$ and $\H_n(\A ,X)$ are strictly flat. Then,
for $n \ge 0$, up to topological
isomorphism, 
\[
\H^n(\A \ptp \B, (X \ptp Y)^*) = 
\displaystyle{\bigoplus_{m+q=n}} [\H_m(\A , X) \ptp \H_q(\B, Y)]^*. \]
\end{corollary}

\begin{proof} By  Theorem \ref{ExternalProductExt},
 up to topological isomorphism, 
$$ 
\H^n(\A \ptp \B, (X \ptp Y)^*) =
H^n (({\mathcal C}_{\sim}(\A, X) \ptp {\mathcal C}_{\sim}(\B, Y))^*).
$$ 
By \cite[Corollary 5.4]{GLW1},
 up to  topological  isomorphism, 
$$ H_n ({\mathcal C}_{\sim}(\A, X) \ptp {\mathcal C}_{\sim}(\B, Y))
=
\displaystyle{\bigoplus_{m+q=n}} 
[H_m({\mathcal C}_{\sim}(\A, X)) \ptp H_q({\mathcal C}_{\sim}(\B, Y))].
$$
Hence, by \cite[Corollary 4.9]{GLW1}, since the
$H_n ({\mathcal C}_{\sim}(\A, X) \ptp {\mathcal C}_{\sim}(\B, Y))$
are Banach spaces,
$$H^n (({\mathcal C}_{\sim}(\A, X) \ptp {\mathcal C}_{\sim}(\B, Y))^*)
\iso [H_n ({\mathcal C}_{\sim}(\A, X) \ptp {\mathcal C}_{\sim}(\B, Y))]^*.$$

Therefore
\[
\H^n(\A \ptp \B, (X \ptp Y)^*) = 
\displaystyle{\bigoplus_{m+q=n}} [\H_m(\A , X) \ptp \H_q(\B, Y)]^*.
\qedhere \]
\end{proof}

\begin{remark} Under the conditions of Corollary \ref{ExternalProductCohom},
 $\H^n(\A \ptp \B, (X \ptp Y)^*)$ is a Banach space, and 
by \cite[Corollary 4.9]{GLW1}, $\H_n(\A \ptp \B, X \ptp Y)$ 
is a Banach space too.
\end{remark}

The closure in $X$ of the linear span of elements of the form 
$\{a \cdot x -x \cdot a:\; a \in \A, x \in X \}$ is denoted by $[X, \A].$


\begin{theorem}\label{Aamen-ExternalProductExt}
Let $\A$ and $\B$  be  Banach  algebras with  bounded approximate 
identities, let $X$  be an essential  Banach 
 $\A$-bimodule and let $Y$ be an essential  Banach  $\B$-bimodule.
Suppose $\A$ is amenable. Then, for $n \ge 0$, up to topological
isomorphism, 
\[ \H^n(\A \ptp \B, (X \ptp Y)^*) = 
\H^n(\B, (X/{[X,\A]}\ptp Y)^*), \]
where $b \cdot (\bar{x} \otimes y) = ( \bar{x} \otimes b \cdot y)$ and
$(\bar{x} \otimes y) \cdot b = ( \bar{x} \otimes y \cdot b)$ for
$\bar{x} \in X/{[X,\A]}$, $y \in Y$ and $b \in \B$.
\end{theorem}

\begin{proof} By assumption $\A$ is biflat, so
$0 \leftarrow \A \stackrel {id_{\A}} {\longleftarrow}  \A \leftarrow 0$ is a flat pseudo-resolution of $\A$ in the category of $\A$-bimodules.
 As in Theorem \ref{ExternalProductExt} we consider the flat pseudo-resolutions 
$0 \leftarrow \A \stackrel {id_{\A}} {\longleftarrow}  \A \leftarrow 0$ and $0 \leftarrow \B \stackrel {\pi_{\B}} {\longleftarrow} \beta(\B )$ of  $\A$ and $\B$ in the categories of bimodules. By
Proposition \ref{ProductResolutions}, their projective tensor product
$\A {\mathbin{\widehat{\otimes}}} \beta(\B )$ is an $\A\ptp\B$-flat  
pseudo-resolution
of $\A {\mathbin{\widehat{\otimes}}}\B$ in $\A \ptp \B$-$\mod$-$\A \ptp \B$.

Similar to Theorem \ref{ExternalProductExt}, one can see that, up to topological isomorphism, 
\begin{eqnarray*} 
\H^n(\A \ptp \B, (X \ptp Y)^*) &=& 
 H^n (((X \ptp Y) {\mathbin{\widehat{\otimes}}}_{(\A \ptp \B)^e}
(\A {\mathbin{\widehat{\otimes}}}\beta(\B )))^*).\\
\end{eqnarray*} 

One can check that, for an amenable Banach algebra $\A$ and  
 for an essential module $X$, up to topological isomorphism,
$X {\mathbin{\widehat{\otimes}}}_{\A^e} \A = X/{[X,\A]}$; see
\cite[Proposition VII.2.17]{He0}. Therefore, by \cite[Proposition II.3.12]{He0}, one can see that 
the following chain complexes are topologically isomorphic:
\begin{eqnarray*} 
(X \ptp Y) {\mathbin{\widehat{\otimes}}}_{(\A \ptp \B)^e}
(\A {\mathbin{\widehat{\otimes}}}\beta(\B )) &\cong & 
(X {\mathbin{\widehat{\otimes}}}_{\A^e} \A ){\mathbin{\widehat{\otimes}}}
(Y {\mathbin{\widehat{\otimes}}}_{\B^e} \beta(\B ))\\
&\cong & 
X/{[X,\A]} \ptp {\mathcal C}_{\sim}(\B, Y).
\end{eqnarray*} 
Thus, up to topological isomorphism, 
\begin{eqnarray*} 
\H^n(\A \ptp \B, (X \ptp Y)^*) &=& 
 H^n (((X \ptp Y) {\mathbin{\widehat{\otimes}}}_{(\A \ptp \B)^e}
(\A {\mathbin{\widehat{\otimes}}}\beta(\B )))^*)\\
 &=& 
H^n ({\mathcal C}_{\sim}(\B, X/{[X,\A]} \ptp Y))^*)\\
&=& \H^n(\B, (X/{[X,\A]}\ptp Y)^*),
\end{eqnarray*} 
where $b \cdot (\bar{x} \otimes y) = ( \bar{x} \otimes b \cdot y)$ and
$(\bar{x} \otimes y) \cdot b = ( \bar{x} \otimes y \cdot b)$ for
$\bar{x} \in X/{[X,\A]}$, $y \in Y$ and $b \in \B$.
\end{proof}

\begin{example} {\rm 
Let $\A$ be  the  Banach algebra $L^1(\RR_+)$ of
complex-valued, 
\newline Lebesgue measurable functions $f$ on $\RR_+$ with
finite $L^1$-norm and
convolution multiplication. In \cite[Theorem 4.6]{GLW2} we showed that 
 all boundary maps of the standard homology complex ${\mathcal C}_{\sim}(\A, \A)$ have closed ranges and that $\H_n(\A , \A)$ and $B_n(\A ,\A)$ are strictly flat in ${\mathcal Ban}$. 
In \cite[Theorem 6.4]{GLW2} we describe explicitly the simplicial homology groups $\H_n(L^1(\RR_+^k),L^1(\RR_+^k))$ and cohomology groups  $\H^n(L^1(\RR_+^k),(L^1(\RR_+^k))^*)$ of the semigroup algebra 
$L^1(\RR_+^k)$.}
\end{example}

\begin{corollary}\label{ho-coho-L1-B} Let 
 $\C$ be an amenable Banach algebra. Then
$$\H_n( L^1(\RR_+^k)\ptp \C, L^1(\RR_+^k)\ptp \C) \iso \{0 \}\;{\rm if} \;\;n>k;$$ 
$$\H^n \left(L^1(\RR_+^k)\ptp \C, \left(L^1(\RR_+^k)\ptp \C \right)^* \right) 
\iso \{0 \}\;{\rm if} \;\;n>k;$$
up to topological isomorphism, 
$$\H_n(L^1(\RR_+^k)\ptp \C, L^1(\RR_+^k)\ptp \C) \iso {\bigoplus\nolimits^{k \choose n}} L^1(\RR_+^k)\ptp \left(\C/[\C, \C] \right)\;{\rm if} \; n\leq k;$$
\noindent and
$$\H^n(L^1(\RR_+^k)\ptp \C, (L^1(\RR_+^k)\ptp \C)^*) \iso
 {\bigoplus\nolimits^{k \choose n}} \left[L^1(\RR_+^k) \ptp \left(\C/[\C,\C] \right) \right]^* $$
if $n\leq k.$
\end{corollary}

\begin{proof} Let $\A=L^1(\RR_+)$. Note that $\A$ and $\C$ have bounded approximate identities.
By \cite[Theorem 5.4]{Ly5}, for an amenable Banach algebra $\C$, 
${\mathcal H}_{n}({\mathcal C}, \C) \iso \{0 \}$ for  all  $n \geq 1$,
${\mathcal H}_0({\mathcal C}, \C) \iso {\mathcal C}/[{\mathcal C}, {\mathcal C}]$. Therefore all boundary maps of the standard homology complex ${\mathcal C}_{\sim}(\C, \C)$ have closed ranges.

In \cite[Theorem 4.6]{GLW2} we showed that 
 all boundary maps of the standard homology complex ${\mathcal C}_{\sim}(\A, \A)$ have closed ranges and that $\H_n(\A , \A)$ and $B_n(\A ,\A)$
are strictly flat in ${\mathcal Ban}$. By \cite[Theorem 4.6]{GLW2},
up to topological isomorphism, the
simplicial homology groups $\H_n(\A,\A)$ are given by 
  $\H_0(\A,\A) \iso \H_1(\A,\A) \iso \A = L^1(\RR_+)$ and
  $\H_n(\A,\A) \iso \{0\} $ for $n\geq 2$.\\

 Note that 
$L^1(\RR_+^k)\ptp \C \iso \A \ptp \B$ where $\B=  L^1(\RR_+^{k-1}) \ptp \C$. We use induction on $k$ to prove the corollary for homology groups. 
For $k=1$, the result follows  from \cite[Theorem 5.5]{GLW2}. The simplicial homology groups $\H_n(\A \ptp \C,\A \ptp \C)$ are given, up to topological isomorphism, by 
 $$ \H_0(\A \ptp \C,\A \ptp \C) \iso  \H_1(\A \ptp \C,\A \ptp \C) \iso \A \ptp \left(\C/[\C, \C] \right)$$ and
 $$ \H_n(\A \ptp \C,\A \ptp \C) \iso \{0\}$$  for $n\geq 2$.

Let $k>1$ and suppose that the result for homology holds for $k-1$. 
As $ L^1(\RR_+^k)\ptp \C \iso \A \ptp \B$ where $\B=  L^1(\RR_+^{k-1}) \ptp \C$, we have
$$\H_n(L^1(\RR_+^k)\ptp \C,L^1(\RR_+^k)\ptp \C) \iso \H_n(\A \ptp \B, \A \ptp \B).$$ 
It also follows from the inductive hypothesis that, 
for all $n$, the $\H_n(\B,\B)$ are Banach spaces and hence the $B_n(\B,\B)$ are closed. 
We can therefore apply \cite[Theorem 5.5]{GLW2} for $\A$ and $\B=L^1(\RR_+^{k-1}) \ptp \C$, where $\C$ is an amenable Banach algebra, to get
\begin{eqnarray*}  \H_n(\A \ptp \B, \A \ptp \B) & \iso & 
 \bigoplus_{m+q=n}\left[\H_m(\A,\A)\ptp \H_q(\B,\B)\right]. \end{eqnarray*}
The terms in this direct sum vanish for $m\geq 2$, and thus we only need 
to consider 
$$\left(\H_0(\A,\A)\ptp \H_n(\B,\B)\right) \oplus 
\left(\H_1(\A,\A)\ptp \H_{n-1}(\B,\B)\right).$$

For cohomology groups, by \cite[Corollary 4.9]{GLW2},
$$\H^n(L^1(\RR_+^k)\ptp \C, (L^1(\RR_+^k)\ptp \C)^*) \iso 
\H^n \left({\mathcal C}_{\sim} \left(L^1(\RR_+^k)\ptp \C, L^1(\RR_+^k) \ptp \C \right)^* \right)$$
$$ \iso {\bigoplus\nolimits^{k \choose n}} \left[L^1(\RR_+^k) \ptp \left(\C/[\C,\C] \right) \right]^* $$
if $n\leq k.$
\end{proof} 
\begin{example}~  {\rm  Some examples of $C^*$-algebras without
non-zero bounded traces are: \\
(i)  The $C^*$-algebra ${\mathcal K}(H)$ of compact operators on  an infinite-dimensional Hilbert space $H$; see \cite[Theorem 2]{An}. We can also show that $ C(\Omega,{\mathcal K}(H))^{tr} = 0$, where $\Omega$ is a compact space. \\
 (ii)  Properly infinite von Neumann
algebras ${\mathcal U}$; see \cite[Example 4.6]{Ly2}. This class includes the   $C^*$-algebra ${\mathcal B}(H)$ of all bounded operators on  an infinite-dimensional Hilbert space $H$; see also \cite{Hal} for the statement ${\mathcal B}(H)^{tr} = 0$.}\\
\end{example}
\begin{example} {\rm Let $\A=\ell^1(\ZZ_+)$ where  $$\ell^1(\ZZ_+)=\left\{
(a_n)_{n=0}^{\infty}: \sum_{n=0}^{\infty} |a_n| < \infty \right\}$$ be 
the unital semigroup Banach algebra of $\ZZ_+$ with  convolution
multiplication and  norm $\left\|(a_n)_{n=0}^{\infty}\right\| = 
\sum_{n=0}^{\infty} |a_n|$. In \cite[Theorem 7.4]{GLW1} we showed that 
 all boundary maps of the standard homology complex ${\mathcal C}_{\sim}(\A, \A)$ have closed ranges and that $\H_n(\A , \A)$ and $B_n(\A ,\A)$ are strictly flat in ${\mathcal Ban}$. 
In \cite[Theorem 7.5]{GLW1} we describe explicitly the simplicial homology groups $\H_n(\ell^1(\ZZ_+^k),\ell^1(\ZZ_+^k))$ and cohomology groups $\H^n(\ell^1(\ZZ_+^k),(\ell^1(\ZZ_+^k))^*)$ of the semigroup algebra 
$\ell^1(\ZZ_+^k)$.}
\end{example}

One can find definitions of the continuous cyclic ${\mathcal H}{\mathcal C}$ and periodic cyclic ${\mathcal H}{\mathcal P}$
homology and cohomology groups of Banach algebras in \cite{Co2}.

\begin{corollary}\label{ho-coho-A-B(H)}
Let ${\mathcal D}$ be a Banach algebra belonging to
one of the following classes:

{\rm (i)} $\D=\ell^1(\ZZ_+^k) \ptp \C$, where 
 $\C$ is a $C^*$-algebra without non-zero bounded traces;

{\rm (ii)} $\D=L^1(\RR_+^k) \ptp \C$, where 
 $\C$ is a $C^*$-algebra without non-zero bounded traces.

Then $ \H_n(\D,  \D)  \iso  \{0\}$
and $ \H^n(\D,  \D) \iso \{0\}$ for all $n \geq 0$;

$${\mathcal H}{\mathcal H}_n(\D) \iso {\mathcal H}{\mathcal H}^n(\D) \iso \{0\}\;\;{\rm for\;\; all}\;\; n\geq 0,$$
$${\mathcal H}{\mathcal C}_n(\D) \iso  {\mathcal H}{\mathcal C}^n(\D) \iso \{0\}\;\;{\rm for\;\; all}\;\; n\geq 0,$$
and
$${\mathcal H}{\mathcal P}_m(\D) \iso {\mathcal H}{\mathcal P}^m(\D) \iso \{0\}\;\;{\rm for}\;\; m=0,1.$$
\end{corollary}

\begin{proof}   
Here we  consider the case with
 $\A= L^1(\RR_+)$ and prove the statement for $\D=L^1(\RR_+^k) \ptp \C$.
By \cite[Theorem 4.1 and Corollary 3.3]{ChSi}, for a
$C^*$-algebra ${\mathcal C}$ without non-zero bounded traces $\H^n({\mathcal C}, {\mathcal C}^*) \iso  \{0\}$ for all   $n \geq 0$.
By \cite[Corollary 1.3]{BEJ1}, $\H_n({\mathcal C}, {\mathcal C}) \iso  \{0\}$ for all   $n \geq 0$.
 Therefore all boundary maps of the standard homology complex ${\mathcal C}_{\sim}(\C, \C)$ have closed ranges.

In \cite[Theorem 4.6]{GLW2} we showed that 
 all boundary maps of the standard homology complex ${\mathcal C}_{\sim}(\A, \A)$ have closed ranges and that $\H_n(\A , \A)$ and $B_n(\A ,\A)$
are strictly flat in ${\mathcal Ban}$. By \cite[Theorem 4.6]{GLW2},
up to topological isomorphism, the
simplicial homology groups $\H_n(\A,\A)$ are given by 
  $\H_0(\A,\A) \iso \H_1(\A,\A) \iso \A = L^1(\RR_+)$ and
  $\H_n(\A,\A) \iso \{0\} $ for $n\geq 2$.

 Note that 
$\D =L^1(\RR_+^k)\ptp \C \iso \A \ptp \B$ where $\B=  L^1(\RR_+^{k-1}) \ptp \C$. We use induction on $k$ to prove the corollary for homology groups. 
For $k=1$, note that $\A$ and $\C$ have bounded approximate identities.
By \cite[Theorem 5.5]{GLW2}, the simplicial homology groups $\H_n(\A \ptp \C,\A \ptp \C)$ are given, up to topological isomorphism, by 
 $$ \H_n(\A \ptp \C,\A \ptp \C) \iso
\displaystyle{\bigoplus_{m+q=n}} [\H_m(\A , \A) \ptp \H_q(\C, \C)] \iso
 \{0\}$$  for $n\geq 0$.

Let $k>1$ and suppose that the result for homology holds for $k-1$. 
As $ L^1(\RR_+^k)\ptp \C \iso \A \ptp \B$ where $\B=  L^1(\RR_+^{k-1}) \ptp \C$, we have
$$\H_n(L^1(\RR_+^k)\ptp \C,L^1(\RR_+^k)\ptp \C) \iso \H_n(\A \ptp \B, \A \ptp \B).$$ 
Note that $\A$ and $\B$ have bounded approximate identities.
Further, it follows from the inductive hypothesis that $\H_n(\B,\B) \iso  \{0\}$ for all $n \geq 0$ and hence the $B_n(\B,\B)$ are closed. 
We can therefore apply \cite[Theorem 5.5]{GLW2} to get
\begin{eqnarray*}  \H_n(\A \ptp \B, \A \ptp \B) & \iso & 
 \bigoplus_{m+q=n}\left[\H_m(\A,\A)\ptp \H_q(\B,\B)\right] \iso
 \{0\} \end{eqnarray*}
 for all $n\geq 0$.

For cohomology groups, by \cite[Corollary 1.3]{BEJ1},
$$\H^n(L^1(\RR_+^k)\ptp \C, (L^1(\RR_+^k)\ptp \C)^*) \iso \{0\}$$
for all $n \geq 0$. 

The triviality of the continuous cyclic ${\mathcal H}{\mathcal C}$
and periodic cyclic ${\mathcal H}{\mathcal P}$
homology and cohomology groups follows from \cite[Corollory 4.7]{Ly3}.
\end{proof} 


\end{document}